\documentclass[sn-mathphys,Numbered]{sn-jnl}

\usepackage{graphicx}%
\usepackage{multirow}%
\usepackage{amsmath,amssymb,amsfonts}%
\usepackage{amsthm}%
\usepackage{mathrsfs}%
\usepackage[title]{appendix}%
\usepackage{xcolor}%
\usepackage{textcomp}%
\usepackage{manyfoot}%
\usepackage{booktabs}%
\usepackage{algorithm}%
\usepackage{algorithmicx}%
\usepackage{algpseudocode}%
\usepackage{listings}%

\theoremstyle{thmstyleone}%
\newtheorem{teorema}{Theorem}[section]
\newtheorem{lema}[teorema]{Lemma}%
\newtheorem{corolario}[teorema]{Corollary}%
\newtheorem{proposicao}[teorema]{Proposition}%
\newtheorem{definicao}[teorema]{Definition}%
\newtheorem{exemplo}[teorema]{Example}%

\newcommand{\R}{{\mathbb R}}
\newcommand{\N}{{\mathbb N}}

\newcommand{\C}{\mathcal{C}}
\newcommand{\F}{\mathbb{F}}

\newcommand{\Ad}{{\rm Ad}}

\newcommand{\g}{\mathfrak{g}}
\renewcommand{\k}{\mathfrak{k}}
\newcommand{\s}{\mathfrak{s}}
\renewcommand{\a}{\mathfrak{a}}

\newcommand{\m}{\mathfrak{m}}

\newcommand{\n}{\mathfrak{n}}

\newcommand{\p}{\mathfrak{p}}

\renewcommand{\sl}{\mathfrak{sl}}

\newcommand{\la}{\lambda}
\newcommand{\T}{\Theta}

\renewcommand{\o}[1]{{\overline{#1}}}

\begin{document}

\title[Control sets on maximal compact subgroups]{Control sets on maximal compact subgroups}

\author*[1]{\fnm{Mauro} \sur{Patr\~ao}}\email{mpatrao@mat.unb.br}

\author[2]{\fnm{La\'ercio} \sur{dos Santos}}\email{laercio.santos@ufjf.br}

\affil*[1]{\orgdiv{Department of Mathematics}, \orgname{Universidade de Bras\'ilia}, \orgaddress{
		 \city{Bras\'ilia}, 
		 \state{DF}, \country{Brazil}}}

\affil[2]{\orgdiv{Department of Mathematics}, \orgname{Universidade Federal de Juiz de Fora}, \orgaddress{
		\city{Juiz de Fora}, 
		\state{MG}, \country{Brazil}}}

\abstract{Let $G$ be a noncompact connected semisimple Lie group, with finite center. In this paper we study the action of semigroups $S$ of $G$, with nonempty interior, acting on maximal compact connected subgroups $K$ of $G$. When $S$ is connected it is well known that the invariant control sets on $K$ are the connected components of $\pi^{-1}(C)$, where $\pi$ is the canonical projection of $K$ onto $F$, $F$ is the flag type of $S$ and $C$ is the only invariant control set for $S$ on $F$. One of the main results of the present paper describes the set of transitivity of a control set, not necessarily invariant, of a semigroup $S$, not necessarily connected, acting on $K$, as fixed points of regular elements in $S$. Furthermore, we show that the number of control sets on $K$ is the product of the number of control sets on the maximal flag manifold of $G$ by the number of invariant control sets on $K$.}

\keywords{Control sets, semisimple Lie groups, maximal compact subgroups, semigroups}


\pacs[MSC Classification]{22E46, 22F30, \textcolor{black}{93B05}}

\maketitle

\section{Introduction}

In this paper we study the action of semigroups with nonempty interior of noncompact connected semisimple Lie groups $G$, with finite center, on their maximal compact connected subgroups. If $G=KAN$ is an Iwasawa decomposition, then the subgroup $K$ identifies with the homogeneous space $G/AN$. The Iwasawa projection of $G$ onto $K$ induces a smooth action of $G$ on $K$ (see Subsection \ref{sec:conjuntocontrolek}) and, thus, $S$ acts on $K$ by restriction. Let us consider the canonical projection $\pi : K \to K/M$, of $K$ onto the maximal flag manifold $K/M$, which is a principal bundle, where $M$ is the centralizer of $A$ in $K$. The action of semigroups, with nonempty interior, on flag manifolds is well known (see for instance \cite{sm93, smt, smo}) and among its applications is included \textcolor{black}{the study of affine control systems given by right invariant vector fields on semisimple Lie groups. In this case, when the control system satisfies the Lie algebra rank condition, the semigroup of the system has nonempty interior and some controllability results can be obtained from semigroup theory (see \cite{j,js,rsms}). Another application is} the generalization of the Selgrade's theorem for linear skew product flows on projective bundles (\cite{selgrade}) to semiflows of endomorphisms of flag bundles (see \cite{msm1}), using the theory of shadowing semigroups developed in \cite{msm0}. In \cite{sms}, it is shown that if $S$ is a connected semigroup with nonempty interior, and $D$ is an invariant control set on $G/AN$, then $D$ is a connected component of $\pi_{\T}^{-1}(C_{\T})$, where $\T=\T(S)$ is the flag type of $S$, $\pi_{\T}:G /AN\to \F_{\T}$ is the canonical projection and $C_{\T}$ is the only invariant control set for $S$ on $\F_{\T}$. Thus, $D$ is diffeomorphic to $C_{\T} \times F_0$ where $F_0$ is a connected component of the fiber $F$ of $\pi_{\T}$.

One of the main results of the present paper  describes the set of transitivity of a control set, not necessarily invariant, of a semigroup with nonempty interior, not necessarily connected, acting on $K \simeq G/AN$, as fixed points of regular elements in the semigroup (see Theorem \ref{teo:controleweylk}). Furthermore, we determine the number of control sets on $K$ (see Theorem \ref{teo:nocontrolsetk}). The main idea here is to use semigroup theory, acting on the maximal flag manifold (see \cite{sm93, smt}), the relationship between control sets on the total space and on the base space, of an associated bundle (see \cite{msm1, msm0}) and the dynamics of hyperbolic elements acting on $K$.

Let us now describe the contents of the paper. In Section \ref{sec:preliminaries}, we introduce some definitions and results that are used here. In Subsection \ref{sec:local-semigroup}, we recall some concepts from \cite{msm1} and \cite{msm0} about the action of local semigroups of continuous maps on topological spaces, the action of local semigroup of local endomorphisms of a principal bundle and how it induces semigroups on the total space, on the base space and on the structure group (which acts on the fiber) of the associated bundle. Finally, we show a result (see Proposition \ref{prop:d0a0}), which establishes a bijection between the control sets on the total space that intersect a fiber over a control set on the base space, with the control sets of the induced semigroup, on the fiber, of the associated bundle. Furthermore, this bijection preserves the order of the control sets. In Subsection \ref{subsec:homspaces}, we turn our attention to homogeneous spaces. First, we describe some objects induced by an admissible triple, which consists of a Cartan involution, a maximal abelian subalgebra and a Weyl chamber. After that, we introduce the canonical objects for objects determined by an admissible triple. Furthermore, we establish a result that decomposes any conjugate of the unipotent subgroup, by an element of the Weyl group, as a product of two subgroups (see Proposition \ref{prop:proposicao2}). Finally, we establish some notations, definitions and results, related to semigroup theory, with nonempty interior, acting on flag manifolds (see \cite{sm93,smt,smo}).

In Section \ref{sec:mainresults}, we state and prove the main results of this paper. First, we obtain some general results using the projection $\pi : K \to K/M$ and some results relating control sets on the total space $K$ with those on the base space $K/M$, along with semigroup theory on $K/M$. After that, we define fixed point components of the type $u$ and study some of their properties, among them, those that relate these components to control sets. After that we do some technical lemmas in preparation for the main results which \textcolor{black}{parameterizes the control sets of $S$ on the compact subgroup $K$ (or on $G/AN$) by elements of the group $U$ (see Theorem \ref{teo:Du}),} describes the transitivity set of a control set as fixed points of regular elements in $S$ (see Theorem \ref{teo:controleweylk}) and state that the number of control sets on $K$ (or on $G/AN$) is the product of the number of control sets on the maximal flag manifold by the cardinality of a quotient group (see Theorem \ref{teo:nocontrolsetk}). Finally, we made examples of the open semigroup of matrices with positive entries, in $\mathrm{Sl}(2, \R)$ and $\mathrm{Sl}(3, \R)$.

\section{Preliminaries}\label{sec:preliminaries}

In this section we introduce some concepts and results that will be used in this paper. For more details see, for instance, \cite{msm1, msm0} (for Subsection \ref{sec:local-semigroup}) and \cite{w, sm93, smt} (for Subsection \ref{subsec:homspaces}).

\subsection{Semigroups and control sets}\label{sec:local-semigroup}

Let $X$ be a topological space. A \emph{local map} on $X$ is a continuous map $\phi:\mathrm{dom}\phi \to X$, where $\mathrm{dom }\phi \subset X$ is an open set. A \emph{local semigroup on $X$} is a family $S$ of local maps that is closed to the composition of maps, in the following sense: if $\phi, \psi \in S$ and $\phi^{-1}(\mathrm{dom}\psi) \neq \emptyset$ then the composition
\[
\psi\circ \phi:\phi^{-1}(\mathrm{dom}\psi) \to X
\]
also belongs to $S$. A local semigroup $S$ on $X$ acts on $X$ by evaluation of maps. For $x \in X$, we denote its
orbit and its backward orbit, respectively, by $Sx$ and $S^*x$, i.e.,
\[
Sx=\{\phi(x):\phi \in S, x \in \mathrm{dom}\phi\} ~\textrm{ and }~ S^*x=\{y: \exists \phi\in S, \phi(y)=x\}.
\]
The action of $S$ on $X$ induces transitive relations $\preceq$, $\preceq_w$ and $\preceq_s$ on $X$, called respectively \emph{algebraic, weak and strong relations} and defined in the following way. Given $x, y \in X$,
\begin{enumerate}
    \item $x \preceq y$ if and only if $y\in Sx$.
    \item $x \preceq_w y$ if and only if $y\in \mathrm{cl}(Sx)$.
    \item $x \preceq_s y$ if and only if $x\in \mathrm{int}(S^{*}y)$.
\end{enumerate}
In general, given a relation $\leq$ on a set, its \emph{symmetrization} $\sim$ is defined by $x\sim y$ if and only if $x\leq y$ and $y\leq x$. We denote by $\sim$, $\sim_w$ and $\sim_s$ the symmetrizations of $\preceq$, $\preceq_w$ and $\preceq_s$, respectively. It is clear that $\sim$, $\sim_w$ and $\sim_s$ are transitive and symmetric but may fail to be reflexive. 
We also denote by $[x]$, $[x]_w$ and $[x]_s$ the classes of $x$ with respect to $\sim$, $\sim_w$ and $\sim_s$, respectively. We have $[x]_s\subset [x]\subset [x]_w$, for all $x \in X$. 
It is not difficult to see that $[x]_s \neq \emptyset$ if and only if $x\sim_sx$. In this case, $x$ is said to be \emph{$S$-self-accessible}.

A \emph{control set of $S$} is a  weak class $[x]_w$ of an element $x\in X$ that is $S$-self-accessible. Let $D=[x]_w\subset X$ be a control set of $S$. If $y \in D$ is self-accessible, then $[y]_s=[x]_s$ (see \cite{msm0}, Corollary 4.7), i.e., in each control set there is a unique strong class, called the \emph{transitivity set} of $D$ and denoted by $D_0$. We have
\[
D_0=\{y \in D: y \text{ is self-accessible}\}.
\]
A control set $D$ is called \emph{$S$-invariant} if $Sx\subset D$, for all $x \in D$. Traditionally, in the context of a semigroup $S$ with nonempty interior of a Lie group $G$, a control set $D$ is a weak class such that there exist $x \in D$ and $g$ in the interior of $S$ with $gx = x$, while its transitivity set is given 
\[
D_0 = \{x \in D : gx = x \mbox{ for some } g \mbox{ in the interior of } S \}
\]
and it can be proven that $D_0$ is dense in $D$ (see \cite{smm}). Since this traditional definition of transitivity set depends only on the interior of $S$, without lost of generality, we can assume that we are dealing with open semigroups of Lie groups. And for open semigropus, it follows that the traditional definition of transitivity set coincides with the one previously presented. 

Let $A$ and $B$ be subsets of $X$ such that $A\subset B$. Then $A$ is called \emph{forward invariant with respect to $B$} if $Sx \cap B \subset A$, for all $x \in A$. If $D$ is a control set, then $D_0$ is forward invariant with respect to $D$ and $D_0$ is dense in $D$ (see \cite{msm0}, Proposition 4.10). The semigroup $S$ is said \emph{accessible at a subset $A\subset X$} if $\mathrm{int}(Sx) \neq \emptyset$ for all $x \in A$. If $A=X$, we say simply that $S$ is \emph{accessible}.

The \emph{order between control sets} is defined in such way that, if $D$ and $D'$ are control sets, then $D\leq D'$ if, and only if, $x\preceq x'$, for some $x\in D_0$ and $x' \in D'_0$.

Let $Q$ be a topological space, $G$ a topological group and $\pi:Q\to X$ a locally trivial principal bundle with structural group $G$. The right action of $G$ on $Q$ is denoted by $(q,a)\in Q\times G \mapsto q\cdot a \in Q$. We denote the fiber over $x \in X$ by $Q_x$.

An associated fiber bundle $E = Q \times _G F \to X$ is constructed via a continuous action of $G$ on the topological space $F$, the typical fiber. The total space $E$ is the quotient $Q \times F/\sim$ where $\sim$ is the equivalence relation $(p, v) \sim (q, w)$ if and only if $q = pa$ and $w=a^{-1}v$. We denote the equivalence class of $(q, v)$ by $q \cdot v \in E$. 

A \emph{local endomorphism of $Q$} is a continuous map $\phi : \mathrm{dom}\phi \to Q$ such that
\begin{enumerate}
    \item $\mathrm{dom}\phi = \pi^{-1}(U)$ where $U \subset X$ is an open set, and
    \item $\phi(qa) = \phi(q) a$ for every $q \in \mathrm{dom}\phi$ and $a \in G$.
\end{enumerate}
We denote by $\mathrm{End}_l(Q)$ the set of all local endomorphisms of $Q$, which is clearly a local semigroup. A mapping $\phi \in \mathrm{End}_l(Q)$ maps fibers into fibers and hence induces a continuous map from $\pi(\mathrm{dom}\phi)$ into $X$, which will be denoted by $\phi_X$. Also, if $E \to X$ is an associated bundle we can define a continuous map by $\phi_E (q \cdot v) = \phi (q) \cdot v$. Its domain is the open set in $E$ over $\pi(\mathrm{dom}\phi)$.

A \emph{local semigroup of endomorphisms $S_Q$} of $Q$ is a local semigroup such that $S_Q \subset \textrm{End}_l(Q)$. $S_Q$ induces local semigroups $S_X$ on $X$ and $S_E$ on $E$, given by
\[
S_X=\{\phi_X:\phi \in S_Q\} \textrm{ and } S_E=\{\phi_E:\phi \in S_Q\}.
\]
We denote by $\pi$ both projections $\pi : E \to X$ and $\pi : Q \to X$. 

For each $x \in X$, we define
\[
S_Q^x = \{\phi \in S_Q : \phi Q_x = Q_x\},
\]
which is a semigroup of homeomorphisms of $Q_x$. After fixing $q \in Q_x$, the action of $S_Q^x$ on the fiber $Q_x$ is given by the following subsemigroup of $G$:
\[
S_q = \{a \in G : qa \in S_Qq\}.
\]
The intersection of $\mathrm{int}(S_Q^*q)$ with the fibers is given by
\[
U_q=\{b \in G: qb \in \mathrm{int}(S_Q^*q)\}.
\]
Let $C$ be a control set of $S_{X}$. The semigroup $S_{Q}$ is said \emph{backward accessible over $C$} if there is $q\in \pi ^{-1}\left( C_{0}\right)$ such that
\begin{equation*}\label{eqacess}
	\mathrm{int}\left( S_{Q}^{*}q\right) \cap \pi ^{-1}\left( C
	\right) \neq \emptyset .
\end{equation*}

\textcolor{black}{We say that the \textit{action of $G$ on $F$ is open} if for every open set $U\subset G$ and every $x \in F$, the orbit $Ux$ is an open set.}

\begin{proposicao}\label{prop:d0a0}
    Let $E = Q \times _{G}F$ such that \textcolor{black}{the action of $G$ on $F$ is open and transitive} and $S_{Q}$ a local semigroup such that its backward orbits are open sets. Also let $C \subset X$ be a control set and $q \in \pi^{-1}(C_{0})$. For every control set $ D\subset E$ of $S_{E}$ with $D_{0}\cap E_{\pi \left( q\right) }\neq \emptyset $, there is a control set $ A_{q}^{D} \subset F$ of $S_{q}$ such that
	\begin{equation*}
		D_{0}\cap E_{\pi \left( q\right) }=q\cdot (A_{q}^{D})_{0}.
	\end{equation*}
	Furthermore, the map $D \mapsto A_{q}^{D}$ is a bijection between the control sets of $S_{E}$ such that $D_{0} \cap E_{\pi \left( q\right) }\neq \emptyset $ and the control sets of $S_{q}$ on $F$, and preserves the order between control sets.
\end{proposicao}

\begin{proof}
    As the backward orbits of $S_{Q}$ are open sets, we have that $U_{q} = S_{q}^{-1}$ (see \cite{msm1}, Lemma 4.7). Furthermore, $S_{Q}$ is backward accessible over $C$, since $\pi(q) \in C_0$ and, thus, there is $\phi \in S_ {Q}$ such that
    \[
    \pi(q)=\phi_X(\pi(q)) = \pi(\phi(q)),
    \]
    which implies that $\phi(q)\in \pi^{-1}(C_0)$ and $\mathrm{int} \left(S_{Q}^*\phi(q)\right) \cap \pi^ {-1}(C) \neq \emptyset$, since $q \in
    S_{Q}^*\phi(q) = \mathrm{int} \left(S_{Q}^*\phi(q)\right)$.
	
    If $D$ is such that $D_{0}\cap E_{\pi(q)}\neq \emptyset $, then there exists a control set $A_{q}^{D}$ of $S_{q}$ such that $ D_{0}\cap E_{\pi \left( q \right)}= q \cdot \left( A_{q}^{D}\right) _{0}$ (see \cite{msm1}, Theorem 4.12). On the other hand, if $A \subset F$ is a control set of $S_{q}$, there is an element of $A_{0}$ fixed by an element of $S_{q} = U_{q}^{-1}$ and therefore fixed by an element of $U_{q}$. By Proposition 4.13 of \cite{msm1}, this implies that there is a control set $D$ such that $ D_{0}\cap E_{\pi \left( q\right) }= q\cdot A_{0}$. Therefore $A_{q}^{D}=A$, which shows the surjectivity of the map. Its injectivity is clear.
	
    To show that this map preserves order let $D$ and $D'$ be control sets such that $D_{0} \cap E_{\pi \left( q \right) }\neq \emptyset$ and $D'_{0} \cap E_{\pi\left( q \right) }\neq \emptyset $. Let $v \in (A_{q}^{D})_0$ and $v' \in (A_{q}^{D'})_0$. If $D \leq D'$, then there is $\phi \in S_Q$ such that $q \cdot v' = \phi_E(q \cdot v) = \phi(q) \cdot v$ (see \cite{msm0}, Proposition 4.11). So, $\phi(q) = qa$ and $v' = av$. This implies that $a \in S_q$ and $A_{q}^{D} \leq A_{q}^{D'}$. Conversely, if $A_{q}^{D} \leq A_{q}^{D'}$, then there exists $a \in S_q$ such that $v' = av$. This implies that there is $\phi \in S_Q$ such that $\phi(q) = qa$ and therefore $q \cdot v' = \phi(q) \cdot v = \phi_E(q \cdot v)$, showing that $D \leq D'$.
\end{proof}

\subsection{Homogeneous spaces}\label{subsec:homspaces}

\textcolor{black}{One of the main differences between the theory of semigroup actions on flag manifolds developed in \cite{smt} and the theory of semigroup actions on maximal compact subgroups developed in the present paper is the following. In flag manifolds, the objects are describe by the Weyl chambers which are in bijection with the classes of $G/MA$, while, in maximal compact subgroups, the objects are describe by the classes of $G/M_0A$, since the connected components of $M$ have a decisive role in these descriptions.} 

\subsubsection{Canonical objects}\label{subsec:canonicalobjects}

Let $\g$ be a real semisimple Lie algebra, $\theta$ a Cartan involution of $\g$ and $\g=\k\oplus \s$ the Cartan decomposition of $\g$ defined by $\theta$, where $\k$ is the subalgebra of fixed points and $\s$ is the eigenspace associated with the eigenvalue $-1$ of $\theta $. Given a maximal abelian subalgebra $\a$ of $\s$, we call $(\theta,\a)$ an \textit{admissible pair} of $\g$. Fixing a Weyl chamber $\a^+ \subset \a$, we call $(\theta,\a,\a^+)$ an \emph{admissible triple}. Let $G$ be a connected Lie group with finite center and Lie algebra $\g$, and $K$ be the connected subgroup of $G$ generated by $\exp(\k)$. An admissible triple $(\theta,\a,\a^+)$ determines the following objects:
$(1)$ the inner product $\langle \cdot, \cdot \rangle_{\theta}$, restricted to $\a$, 
$(2)$ the maximal abelian subgroup $A=\exp\a$, 
$(3)$ the normalizer $M_*A$ and the centralizer $MA$, of $\a$ in $G$, where $M_*$ and $M$ are the normalizer and the centralizer of $\a$ in $K$,
$(4)$ the quotient groups $U=M_*A/M_0A \simeq M_*/M_0$ \label{def:grupoU(B)} and $C=MA/M_0A\simeq M/M_0$, where $M_0$ is the identity component of $M$ and $M_*$, 
$(5)$ the Weyl group $W$, 
$(6)$ the system of (co-)roots $\Pi$ ($\Pi_{\a}$),  
$(7)$ the simple system of (co-) roots $\Sigma$ ($\Sigma_{\a}$) and the set of positive (co-)roots $\Pi^+$ ($\Pi_{\a}^+$), 
$(8)$ for each $\T \subset \Sigma$, the parabolic subalgebra $\p_{\T}$ and the parabolic subgroup $P_{\T}$, and 
$(9)$ the components of the Iwasawa decompositions of $\p_{\T}$ and $P_{\T}$. Since $G$ has finite center, it follows that $K$, $M_*$ and $M$ are compact subgroups and $U$ and $C$ are finite groups. We have that $M$ is a normal subgroup of $M_*$ and $C$ is a normal subgroup of $U$. 
\textcolor{black}{Let us assume that $G$ has a complexification to guarantee that $C$ is an abelian group (see \cite{knapp}, Theorem 7.53). The above assumptions hold for all linear groups.}

The group $G$ acts on itself by conjugation, on its Lie algebra through adjoint action and on the dual of the algebra by coadjoint action. The following notation will be used for the adjoint action of $G$ on these objects
\begin{enumerate}[$(i)$]
    \item $gwg^{-1} = \Ad(g)w\Ad(g^{-1})$ for $w \in W$,
    \item $g\alpha = \Ad(g)^*\alpha$ for $\alpha \in \Pi$, and
    \item $gX = \Ad(g)X$ for $X \in \g$.
\end{enumerate}

The \emph{sets of maximal abelian subgroups and Weyl chambers} in $G$ are defined, respectively, by
\begin{equation*}
    \mathcal{A} = \{\exp(\a) : (\theta,\a) \,\, \mbox{admissible pair}\}
\end{equation*}
and
\begin{equation*}
    \mathcal{C} = \{\exp(\a^+) : (\theta,\a,\a^+) \,\, \mbox{admissible triple}\}.
\end{equation*}
The group $G$ acts on $\mathcal{A}$ and $\C$ by conjugation. The isotropy subgroups of $A=\exp \a \in \mathcal{A}$ and of $A^+=\exp \a^+ \in \C$ are $M_*A$ and $MA$, respectively. Thus, the sets $\mathcal{A}$ and $\C$ are identified with the homogeneous spaces $G/M_*A$ and $G/MA$, respectively. Since the subgroup $M_0A$ is the identity component of both $M_*A$ and $MA$ it follows that $M_0A$ is a normal subgroup of $M_*A$ and $MA$. Therefore, the canonical projections
\begin{equation}\label{eq:fibabeliano}
    G/M_0A \to G/M_*A, ~~gM_0A \mapsto gM_*A
\end{equation}
and
\begin{equation}\label{eq:fibcamara}
    G/M_0A \to G/MA, ~~gM_0A \mapsto gMA,
\end{equation}
are principal bundles with structure groups $U$ and $C$, respectively. From the projections on (\ref{eq:fibabeliano}) and on (\ref{eq:fibcamara}), each coset $gM_0A$ determines the cosets $gM_*A$ and $gMA$ which, in turn, are identified respectively with the maximal abelian subgroup $gAg^{-1}=\exp g\a$ and with the Weyl chamber $gA^+g^{-1}=\exp g\a^+$. 
Consider the notation $B=M_0A$. The Weyl chamber determined by $gB$ will be denoted by $\la(gB)$, i.e., $\la(gB)=gA^+g^{-1}=\exp g\a^+$.

\begin{lema}\label{lema:aplicacoeslambda}
    If $(\theta,\a,\a^+)$ and $(\o{\theta},\o{\a},\o{\a}^+)$ are admissible triple such that $\a^+=\o{\a}^+$, then they determine the same objects $(1)-(9)$ mentioned above. Thus, these objects are determined by each coset $gB$, $g \in G$, and will be denoted by the juxtaposition of $(gB)$ on the right hand side. In addition, given $g \in G$, we have
    \begin{enumerate}[$(i)$]
        \item\label{enum:objconjprimeiro} $\a(g B) = g\a(B)$ and the same for $\m$, $\n$, $\Pi$, $\Sigma$, $\Pi_{\a}$, and $\Sigma_{\a}$, 
	\item $\langle gH, g\widetilde{H} \rangle({gB}) = \langle H, \widetilde{H} \rangle({B})$, for all $ H, \widetilde{H} \in         \a(B)$,
        \item $(M_*A)(gB) = g(M_*A)(B)g^{-1}$ and the same for $A$, $M$, $M_0$ and $N$,
	\item $g\T(B)=\T(gB)$, for each $\T(B)\subset \Sigma(B)$,
	\item $\p_{g\T}(gB) = g\p_{\T}(B)$, for each $\T \subset \Sigma(B)$, and
	\item\label{enum:objconjultimo} $U(gB)=gU(B)g^{-1}$ and the same for $C$ and $W$.
    \end{enumerate}
\end{lema}

\begin{proof}
    Equalities $(\ref{enum:objconjprimeiro}) - (\ref{enum:objconjultimo})$ are easy to check by replacing $B$ by the admissible triple $(\theta, \a, \a^+)$ and $gB$ by the admissible triple
    \[
    (\Ad(g) \theta \Ad(g)^{-1}, g\a, g\a^+).
    \]
    Since Cartan involutions are all conjugate, there exists $g \in G$ such that
    \[
    (\o{\theta},\o{\a},\o{\a}^+) = (\Ad(g) \theta \Ad(g)^{-1}, g\a, g\a^+)
    \]
    and, thus, $g\a^+ =  \o{\a}^+ =\a^+$. So $g \in MA$ and therefore $\o{\a} = g\a =\a$. This implies, using relations $(\ref{enum:objconjprimeiro}) - (\ref{enum:objconjultimo})$, that $(\theta,\a,\a^+)$ and $(\o{\theta},\o{\a},\o{\a}^+)$ determine the same objects $(2)-(9)$ mentioned above. To get the same for the object $(1)$, just note that if $g \in MA$ and $ H, \widetilde{H} \in \a(B)$, then
    \begin{equation*}
        \langle H, \widetilde{H} \rangle_{g \theta g^{-1}} = \langle gH,
	g\widetilde{H} \rangle_{g \theta g^{-1}} = \langle H,
	\widetilde{H} \rangle_{\theta}.
    \end{equation*}
    Therefore these objects are determined by $\a^+$ and, consequently, by $\exp(\a^+) \in \C$.
\end{proof}

Defining
\begin{equation}\label{eq:preU}
	\mathcal{U} = \{(g'B, u) : g' \in G, u \in U(g'B)\},
\end{equation}
we have that the adjoint action of $G$ on $\mathcal{U}$, given by
\begin{equation}
    g(g'B, u) = (g g'B, g u g^{-1})
\end{equation}
where $g \in G$, is well defined since by Lemma \ref{lema:aplicacoeslambda}, $g u g^{-1} \in U(g g'B)$. \textcolor{black}{From now on, the group $U$ defined on Page \pageref{def:grupoU(B)} will be denoted by $U(B)$ and we will denote by $U$ the quotient group of $\mathcal{U}$ for this action of $G$} and called \emph{canonical group $U$ of $\g$}. An element of $U$ is an orbit of a pair $(g'B, u)$ denoted by $[(g'B, u)]$. The next proposition justifies the nomenclature
used.

\begin{proposicao}\label{prop:Ucanonico}
    For each $g \in G$ and $u \in U$, there is a unique $u(gB) \in U(gB)$ such that $u = [(gB, u(gB))]$. For every $g \in G$, the map $u \mapsto u(gB)$ is a canonical isomorphism between the canonical group $U$ and the group $U(gB)$, where for all $u, \widetilde{u} \in U$, the product given by
    \begin{equation}\label{eq:produtoU}
        u\widetilde{u} = [(gB,u(gB)\widetilde{u}(gB))]
    \end{equation}
    and the inverse given by
    \begin{equation}\label{eq:inversaU}
	u^{-1} = [(gB, u(gB)^{-1})]
    \end{equation}
    are well defined. Furthermore, for all $g,g' \in G$, we have that
    \[
    u(gg'B) = g u(g'B) g^{-1}.
    \]
\end{proposicao}

\begin{proof}
    Let $h \in G$ and $s = [(hB, \widetilde{s})] \in U$. By definition, if $\widehat{s} \in U(hB)$ is such that $s =[(hB, \widehat{s})]$, then there exists $g \in G$ such that $hB = ghB $ and $\widehat{s} = g \widetilde{s} g^{-1}$. Since $h^{-1}ghB=B$ there is $m \in B$ such that $h^{-1}gh=m$, i.e.,
    \[
    g=hmh^{-1} \in hBh^{-1}=(M_0A)(hB).
    \]
    Hence, writing $\widetilde{s}=\widetilde{s}_*(M_0A)(hB)$, with $\widetilde{s}_*\in (M_*A)(hB)$, we have
    \[
    \widehat{s} = g \widetilde{s}g^{-1}=\widetilde{s}_*(\widetilde{s}_*^{-1}g \widetilde{s}_*)(M_0A)(hB)g^{-1}=\widetilde{s}_*(M_0A)(hB)=\widetilde{s},
    \]
    since $(M_*A)(hB)$ normalizes $(M_0A)(hB)$ and thus $\widetilde{s}_*^{-1}g \widetilde{s}_* \in (M_0A)(hB)$. Therefore, $\widetilde{s}$ depends only on  $s \in U$ and the coset $hB$, and will be denoted by $s(hB)$. For every $g \in
    G$, we have that
    \[
    (ghB, g s(hB) g^{-1}) \in \mathcal{U},
    \]
    and $s = [(ghB, g s(hB) g^{-1})]$, showing that $s(ghB)$ is well defined and equals $g s(hB) g^{-1}$. If $u \in U$, then $u = [(g'B, u(g'B))]$, for some $g' \in G$. Let $g \in G$ such that $gg'B=hB$. So
    \[
    u=[(gg'B, gu(g'B)g^{-1})]=[(hB,u(hB))].
    \]
    Hence, $u(hB)$ is defined for every $u \in U$.
	
    To show that the product and inverse given by (\ref{eq:produtoU}) and  (\ref{eq:inversaU}), respectively, are well defined, let $g,g' \in G$. Let $h \in G$ such that $gB = hg'B$. Therefore, as $u(hg'B) = hu(g'B)h^{-1}$, for all $u \in U$, we have that
    \begin{equation*}
	[(gB, u(gB)\widetilde{u}(gB))] = [(hg'B, h u(g'B) \widetilde{u}(g'B) h^{-1})] = [(g'B, u(g'B)\widetilde{u}(g'B))]
    \end{equation*}
    and
    \begin{equation*}
	[(gB, u(gB)^{-1})] = [(hg'B, h u(g'B)^{-1} h^{-1})] = [(g'B, u(g'B)^{-1})].
    \end{equation*}
    This shows that the map $u \mapsto u(gB)$ is clearly a homomorphism of the group $U$ in the group $U(gB)$. Furthermore, it is injective, since if $u(gB) = s(gB)$, then
    \begin{equation*}
	u = [(gB, u(gB))] = [(gB, s(gB))] = s.
    \end{equation*}
    On the other hand, it is also surjective, since, if  \textcolor{black}{$\widetilde{u}$ is a given element in $U(gB)$, defining $u = [(gB, \widetilde{u})]$, it follows that $u(gB)=\widetilde{u}$}.
\end{proof}

Similarly, canonical objects are constructed for other objects determined by cosets $gB$, $g \in G$.

Given $u \in U$ and $g\in G$, denote by $N_u(gB)$\label{def:Nu} the subgroup defined by
\begin{equation}\label{eq:defNu}
    N_u(gB)=u(gB)N(gB)u(gB)^{-1}=u_*N(gB)u_*^{-1},
\end{equation}
where $u_* \in (M_*A)(gB)$ is such that $u(gB)=u_*(M_0A)(gB)$.
The second equality in (\ref{eq:defNu}) is well defined since $(M_0A)(gB)$ normalizes $N(gB)$.
By definition, if $1$ denotes the identity of $U$, then $N_1(gB)=N(gB)$.
\textcolor{black}{We note, by equality (\ref{eq:defNu}), Proposition \ref{prop:Ucanonico} and Lemma \ref{lema:aplicacoeslambda}, that
\[
N_u(gB)=u(gB)N(gB)u(gB)^{-1}=gu(B)N(B)u(B)^{-1}g^{-1}=gN_u(B)g^{-1}.
\]
}

\textcolor{black}{If $g \in G$, $u \in U$ and $u(gB)=u'(M_0A)(gB)$, $u'\in (M_*A)(gB)$, then the set
\[
\{u'mgB:m \in (M_0A)(gB)\}
\]
is unitary. In fact, as $(M_*A)(gB)=gM_*Ag^{-1}$ and $(M_0A)(gB)=gM_0Ag^{-1}$, from Lemma \ref{lema:aplicacoeslambda}, we have that if $m \in (M_0A)(gB)$, then there are $u_* \in M_*$, $a \in A$ and $\widetilde{m} \in M_0A$ such that $u'=gu_*ag^{-1}$ and $m=g\widetilde{m}g^{-1}$. So
\[
u'mgB=(gu_*ag^{-1})(g\widetilde{m}g^{-1})gB=gu_*a\widetilde{m}B=gu_*B.
\]
We denote the only element of this set by $u(gB)gB$.
The above argument shows the following result.}

\begin{lema}\label{lema:nu}
    If $g \in G$, $u \in U$ and $u(B)=u_*(M_0A)(B)$, \textcolor{black}{$u_*\in M_*$}, then 
    \begin{enumerate}
        \item $N_u(gB)=gN_u(B)g^{-1}$.
	\item $u(gB)gB=gu_*B$.
    \end{enumerate}
\end{lema}

Let $u \in U$. Given $g \in G$, let $u_* \in (M_*A)(gB)$ and $w \in W$ such that $u(gB)=u_*(M_0A)(gB)$ and
\[
w(gB)=\mathrm{Ad}(u_*)|_{\a(gB)}.
\]
The element $w \in W$ above is denoted by $w=\mathrm{Ad}(u)|_{\a}$ \label{def:w=Adu}. 
Conversely, given $w \in W$ there is $u \in U$ such that $w=\mathrm{Ad}(u)|_{\a}$. In fact, for every $g \in G$, there exists $u_*\in (M_*A)(gB)$ such that
\[
w(gB)=\mathrm{Ad}(u_*)|_{\a(gB)}.
\]
Taking $u \in U$ such that $u(gB)=u_*(M_0A)(gB)$, by Proposition \ref{prop:Ucanonico}, we have that $w=\mathrm{Ad}(u)|_ {\a}$.

\subsubsection{Decomposition of the unipotent component}\label{sec:decunip}

Given a root $\alpha \in \Pi$, let $l(\alpha)$ denote the height of $\alpha$. For every $k \in \N$, let 
\[
\Pi^k=\{\alpha \in \Pi^+:l(\alpha)=k\} \text{ and } \Pi_l=\bigcup_{k\geq l}\Pi^k.
\]
By definition, $\Pi^1=\Sigma$ and $\Pi_1=\Pi^+$. Let us also consider
\[
\n_l=\sum_{\alpha\in \Pi_l}\g_{\alpha}  ~\text{ and }~ N_l=\exp(\n_l).
\]
As $[\n_l,\n_{l'}]\subset \n_{l+l'}$, for all $l, l' \in \N$, we have that $\n_l$ is a  subalgebra. Furthermore, if $\Delta_l \subset \Pi^l$, then
\[
\n_{\Delta_l}=\left(\sum_{\alpha \in \Delta_l}\g_{\alpha}\right) \oplus \n_{l+1}
\]
is subalgebra. Consider the notation $N_{\Delta_l}=\exp(\n_{\Delta_l})$.

\begin{lema}\label{lema:lema1}
    If $\Delta_l \subset \Pi^l$ and $\alpha \in \Delta_l$, then $N_{\Delta_l}=\exp(\g_{\alpha})N_{\Delta_l-\{\alpha\}}$.
\end{lema}

\begin{proof}
    As $[\g_{\alpha},\n_{\Delta_l-\{\alpha\}}]\subset \n_{l+1}\subset \n_{\Delta_l-\{\alpha\}}$ it follows that $\exp(\g_{\alpha})N_{\Delta_l-\{\alpha\}}$ is a subgroup with Lie algebra equal to $\g_{\alpha}\oplus \n_{\Delta_l-\{\alpha\}}=\n_{\Delta_l}$. Therefore, $N_{\Delta_l}$ is contained in $\exp(\g_{\alpha})N_{\Delta_l-\{\alpha\}}$. The other inclusion follows from the equality $\n_{\Delta_l}=\g_{\alpha}\oplus \n_{\Delta_l-\{\alpha\}}$.
\end{proof}

\begin{proposicao}\label{prop:proposicao1}
    If $\Pi^+=\Delta \dot{\cup}\Gamma$, then $\displaystyle N=\prod_{\alpha\in \Delta} \exp(\g_{\alpha})\prod_{\alpha\in \Gamma} \exp(\g_{\alpha})$.
\end{proposicao}

\begin{proof}
    For each $\alpha \in \Pi^1$, we write, by Lemma \ref{lema:lema1},
    \[
    N = \left\{
    \begin{array}{ccc}
        \exp(\g_{\alpha})N_{\Pi_1-\{\alpha\}} & \text{if}  & \alpha \in \Delta\\
        N_{\Pi_1-\{\alpha\}}\exp(\g_{\alpha}) & \text{if}  & \alpha \in \Gamma.
    \end{array}
    \right.
    \]
    For $\beta \in \Pi^1-\{\alpha\}$ we write, by Lemma \ref{lema:lema1},
    \[
    N_{\Pi_1-\{\alpha\}} = \left\{
    \begin{array}{ccc}
        \exp(\g_{\beta})N_{\Pi_1-\{\alpha,\beta\}} & \text{if}  & \beta \in \Delta\\
	N_{\Pi_1-\{\alpha,\beta\}}\exp(\g_{\beta}) & \text{if}  & \beta \in \Gamma.
    \end{array}
    \right.
    \]
    So,
    \[
    N= \left\{
    \begin{array}{ccl}
        \exp(\g_{\alpha})\exp(\g_{\beta})N_{\Pi_1-\{\alpha,\beta\}} & \text{if}  & \alpha, \beta \in \Delta\\
	\exp(\g_{\alpha})N_{\Pi_1-\{\alpha,\beta\}}\exp(\g_{\beta}) & \text{if}  & \alpha \in \Delta, \beta \in \Gamma\\
	\exp(\g_{\beta})N_{\Pi_1-\{\alpha,\beta\}}\exp(\g_{\alpha}) & \text{if}  & \alpha \in \Gamma, \beta \in \Delta\\
	N_{\Pi_1-\{\alpha,\beta\}}\exp(\g_{\beta})\exp(\g_{\alpha}) & \text{if}  & \alpha, \beta \in \Gamma.\\
    \end{array}
    \right.
    \]
    Continuing, we obtain the result.
\end{proof}

The Weyl group $W$ acts on $N$ as follows: given $w\in W$, let $u \in U$ such that $w=\mathrm{Ad}(u)|_{\a} $ (see definition on page \pageref{def:w=Adu}). We define
\[
wNw^{-1}=u(B)Nu(B)^{-1}=N_u(B).
\]
Since $(MA)(B)$ normalizes $N$ it follows that this action is well defined. For $w \in W$, let
\[
\Pi_w^+ =\Pi^+\cap w^{-1}\Pi^-
\]
be the set of positive roots that are taken into negative by $w$.

\begin{proposicao}\label{prop:proposicao2}
    $N_u(B)=(N_u(B)\cap N^{-})(N_u(B)\cap N)$.
\end{proposicao}

\begin{proof}
    Let $w=\mathrm{Ad}(u)|_{\a}$. We have, by definition and from Proposition \ref{prop:proposicao1}, that
    \begin{eqnarray*}
        N_u(B)&=&wNw^{-1}\\
	&=&\prod_{\alpha \in \Pi^+_w}w\exp(\g_{\alpha})w^{-1} \prod_{\alpha \in \Pi^+-\Pi^+_w}w\exp(\g_{\alpha})w^{-1}\\
	&=&\prod_{\alpha \in \Pi^+_w}\exp(w\g_{\alpha}) \prod_{\alpha \in \Pi^+-\Pi^+_w}\exp(w\g_{\alpha})\\
	&=&\prod_{\alpha \in \Pi^+_w}\exp(\g_{w^*\alpha}) \prod_{\alpha \in \Pi^+-\Pi^+_w}\exp(\g_{w^*\alpha})\\
	&\subset& (wNw^{-1}\cap N^-)(wNw^{-1}\cap N)\\
	&=&(N_u(B)\cap N^{-})(N_u(B)\cap N).
    \end{eqnarray*}
    The other inclusion is immediate.
\end{proof}

\subsubsection{Control sets on flag manifolds}

Given $\T \subset \Sigma$, the \emph{parabolic subalgebra $\p_{\T}(gB)$ of type $\T$ determined by $gB$} is, by definition, the parabolic subalgebra $\p_{\T(gB)}(gB)$ and the \emph{parabolic subgroup $P_{\T}(gB)$ of type $\T$} is $P_{\T(gB)}(gB)$. The \emph{(generalized) flag manifold $\Bbb{F}_{\T}$ of type $\T$} is the set
\begin{equation*}
    \Bbb{F}_{\T} = \{\p_{\T}(gB) : g \in G\}
\end{equation*}
of all parabolic subalgebras of type $\T$. If $\T = \emptyset$, the flag manifold $\Bbb{F} = \Bbb{F}_{\emptyset}$ is called
{\it  maximal (generalized) flag manifold of $\g$}. The adjoint action of
$G$ on $\Bbb{F}_{\T}$, given by $g'\p_{\T}(gB)$, is well defined, because, by Lemma \ref{lema:aplicacoeslambda}, we have that
\begin{equation*}\label{eqacaoGemFT}
    g'\p_{\T}(gB) = g'\p_{\T(gB)}(gB) = \p_{\T(g'gB)}(g'gB) = \p_{\T}(g' gB).
\end{equation*}
The flag manifold $\F_{\T}$ is diffeomorphic to the homogeneous space $G/P_{\T}(gB)$.

Let $g \in G$. For every $h$ in the Weyl chamber $\la(gB)$, the set of fixed points of the action of $h$ on $\F$ is given by
\[
\{w(gB)\p(gB):w \in W\}.
\]
Furthermore, $\p(gB)$ is the only attractor of $h$ and its stable manifold $N(gB)\p(gB)$ is open and dense in $\F$. For each $w\in W$ the fixed point $w(gB)\p(gB)$ is denoted by $\p_w(gB)$ and called \emph{fixed point of type $w$ in $\F$ determined by $gB$}. For every $g \in G$ we have
\begin{equation*}\label{eq:gpw}
    g\p_w(B)=\p_w(gB).
\end{equation*}

Let $S\subset G$ be an open semigroup. The \emph{sets of attractors $\mathrm{atr}(S)$ and fixed points $\mathrm{fix}_w(S)$ of type $w$ of $S$}, in $\F$, are defined by
\[
\mathrm{atr}(S)=\{\p(gB): \la(gB) \cap S \neq \emptyset\}
\]
and
\[
\mathrm{fix}_w(S)=\{\p_w(gB): \la(gB) \cap S \neq \emptyset\}.
\]
In particular, $\mathrm{fix}_1(S)=\mathrm{atr}(S)$.

The following result characterizes the control sets of $S$ on the maximal flag manifold $\F$ of $\g$ (see \cite{smt, sm93}).

\begin{proposicao}\label{prop:controleweyl}
    For every $w \in W$, there exists a control set $\Bbb{C}(w)$ on $\F$ such that
    \begin{equation*}
        \Bbb{C}(w)_0 = \mathrm{fix}_w(S)
    \end{equation*}
    and these are all control sets of $S$ on $\Bbb{F}$. There is a unique invariant control set $\Bbb{C} = \Bbb{C}(1)$ and $\Bbb{C}_0 = \mathrm{atr}(S)$. Furthermore, the number of
    control sets of $S$ on $\Bbb{F}$ is the cardinality of $W(S) \backslash W$, where
    \[
    W(S)=\{w \in W:\Bbb{C}(w)=\Bbb{C}\}
    \]
    is a subgroup of $W$ and there exists a subset $\T(S)$ of $\Sigma$ such that $W(S)=W_{\T(S)}$. The \emph{flag type of $S$} is by definition the subset $\T(S)$ or the flag manifold $\F_{\Theta(S)}$.
\end{proposicao}

\section{Main Results}\label{sec:mainresults}

In this section we enunciate and demonstrate the main results of this paper. For this, as explained in Subsection \ref{sec:local-semigroup}, we can assume that $S\subset G$ is an open semigroup and that $\la_0$ is a fixed Weyl chamber in $G$ such that $S \cap \la_0 \neq \emptyset$. We consider the subgroups $M$, $A$ and $N$, determined by $\la_0$.

\subsection{Projection on maximal flag manifold}\label{sec:conjuntocontrolek}

\textcolor{black}{In this subsection, we relate the control sets of $S$ on maximal compact subgroup $K$ with the control sets of $S$ on the maximal flag manifold $K/M$.} We denote by $\kappa$ the projection of $G$ onto $K$ given by the Iwasawa decomposition, i.e.,
\[
\kappa : G \to K, ~~ g\mapsto \kappa(g), ~\text{such that}~ g \in \kappa(g)AN
\]
and by $\delta$ the map
\begin{equation*}\label{eq:acaoG-K}
    \delta : G \times K \to K, ~~ (g,k)\mapsto \kappa(gk).
\end{equation*}
Since $A$ normalizes $N$ it follows that $\delta$ is an action of $G$ on $K$. We denote by
\[
\pi:K \to K/M
\]
the canonical projection of $K$ onto the maximal flag manifold $\F\simeq K/M$, which is a principal bundle with structure group $M$. The action of each $g \in G$ maps fibers to fibers, since $M$ normalizes $AN$, i.e., each element $g \in G$ acts as a bundle endomorphism. The semigroup $S$ acts on $K$ by restricting the action of $G$.

\begin{proposicao}\label{prop:kpreimcw}
    If $w \in W$ and $k \in \pi^{-1}(\Bbb{C}(w)_0)$, then
    \begin{enumerate}
	\item $S_k$ is an open subgroup of $M$.
	\item There is a control set $\Bbb{D} \subset K$ such that $k \in \Bbb{D}_0 \subset \pi^{-1}(\Bbb{C}(w)_0)$.
	\item $k\cdot M_0 \subset \Bbb{D}_0$.
	\item $\Bbb{D}_0$ is forward invariant with respect to $\pi^{-1}(\Bbb{C}(w)_0)$.
	\item If $w \in W(S)$, then $\Bbb{D}$ is $S$-invariant.
    \end{enumerate}
\end{proposicao}

\begin{proof}
    Take $k \in \pi^{-1}(\Bbb{C}(w)_0)$. We have that $S_k=U_k^{-1}$ is an open nonempty semigroup of $M$ (see \cite{msm1}, Lemmas 4.7 and 4.10). As $M$ is compact it follows that $S_k$ is a subgroup containing $M_0$ and $S_k$ is an invariant control set of $S_k$. By Proposition \ref{prop:d0a0}, there exists a control set $\Bbb{D} \subset K$ for the semigroup $S$ such that
    \[
    \Bbb{D}_0\cap K_{\pi(k)} = k \cdot S_k.
    \]
    As $M_0 \subset S_k$, it follows that $k\cdot M_0 \subset \Bbb{D}_0$ and in particular $k \in \Bbb{D}_0$. So $\pi(\Bbb{D}_0) \subset \Bbb{C}(w)_0$ (see \cite{msm1}, Lemma 4.1) which shows items 1, 2 and 3.
	
    For statements 4 and 5, let $k' \in \Bbb{D}_0$. By Proposition \ref{prop:d0a0}, there exists a control set $\Bbb{A} \subset M$ for the semigroup $S_{k'}$ such that
    \[
    \Bbb{D}_0\cap K_{\pi(k')} = k' \cdot \Bbb{A}_0.
    \]
    Since $S_{k'}$ is a subgroup it follows that $\Bbb{A}$ is $S_{k'}$-invariant and thus $\Bbb{D}_0$ is forward invariant with respect to $\pi^{ -1}(\Bbb{C}(w)_0)$ (see \cite{msm1}, Proposition 4.14). Finally, if $w \in W(S)$, then by Corollary 4.16 of \cite{msm1} and Corollary 4.19 of \cite{msm0}, $\Bbb{D}$ is $S$-invariant.
\end{proof}

\begin{proposicao}\label{prop:pid0c0}
    If $\Bbb{D}\subset K$ is a control set, then there exists $w \in W$ such that $\pi(\Bbb{D}_0)=\Bbb{C}(w)_0$.
\end{proposicao}

\begin{proof}
    By Lemma 4.1 of \cite{msm1} and Proposition \ref{prop:controleweyl}, there exists $w \in W$ such that $\pi(\Bbb{D}_0)\subset \Bbb{C}(w)_0$. Let $x \in \Bbb{C}(w)_0$ and $k \in \Bbb{D}_0$. Thus, there exists $g \in S$ such that $\pi(gk)=g\pi(k)=x$ (see \cite{msm0}, Proposition 4.8), that is, $gk \in \pi^{ -1}(\Bbb{C}(w)_0)$. Hence, $gk \in \Bbb{D}_0$, since $\Bbb{D}_0$ is forward invariant with respect to $\pi^{-1}(\Bbb{C}(w)_0) $, by Proposition \ref{prop:kpreimcw}. Therefore, $x \in \pi(\Bbb{D}_0)$.
\end{proof}

\begin{corolario}\label{cor:ccswfinito}
    If $w \in W$, then the number of control sets of $S$ contained in $\pi^{-1}(\Bbb{C}(w))$ is finite. In particular, the number of control sets of $S$ on $K$ is finite.
\end{corolario}

\begin{proof}
    Let $k \in \pi^{-1}(\Bbb{C}(w)_0)$. By the Propositions \ref{prop:d0a0} and \ref{prop:pid0c0} there is a bijection between the control sets of $S$ contained in $\pi^{-1}(\Bbb{C}(w))$ and the control sets of $S_k$, on fiber $M$. As $S_k\subset M$ is a subgroup, by Proposition \ref{prop:kpreimcw}, and $M$ is compact, it follows that $S_k$ has a finite number of control sets. Therefore, the number of control sets of $S$ contained in $\pi^{-1}(\Bbb{C}(w))$ is finite. For the second statement, just use the first one, that the number of control sets of $S$ on $\F$ is finite and Lemma 4.1 of \cite{msm1}.
\end{proof}

\begin{proposicao}\label{prop:ccik}
    Let $\Bbb{D} \subset K$ be a control set.
    \begin{enumerate}
	\item If $\Bbb{D}$ is $S$-invariant, then $\pi(\Bbb{D})=\Bbb{C}$.
	\item If $\Bbb{D}\subset \pi^{-1}(\Bbb{C})$, then $\Bbb{D}$ is $S$-invariant.
    \end{enumerate} 
\end{proposicao}

\begin{proof}
    Item 1 follows from Lemma 4.5 of \cite{msm1} and the fact that $\Bbb{C}$ is the only invariant control set of $S$ on $\F$. For item 2, since $\pi(\Bbb{D})\subset \pi(\pi^{-1}(\Bbb{C})) \subset \Bbb{C}$ we have that $\pi (\Bbb{D}_0)\subset \Bbb{C}_0$ (see \cite{msm1}, Lemma 4.1). By Proposition \ref{prop:kpreimcw}, $\Bbb{D}$ is $S$-invariant.
\end{proof}

\begin{proposicao}\label{prop:Dm}
    If $\Bbb{D}\subset K$ is a control set, then for all $m \in M$, $\Bbb{D}m$ is a control set and $(\Bbb{D}m )_0=\Bbb{D}_0m$. Furthermore,
    \begin{enumerate}
	\item $\Bbb{D}$ is $S$-invariant if and only if $\Bbb{D}m$ is $S$-invariant.
	\item $\Bbb{D}m=\Bbb{D}$, for all $m \in M_0$.
    \end{enumerate}
\end{proposicao}

\begin{proof}
    Take $k \in \Bbb{D}_0$. One has that $\Bbb{D}=\mathrm{cl}(Sk) \cap \mathrm{int}(S^{-1}k)$ (see \cite{msm0}, Proposition 4.8). Since right translation by $m$ is a homeomorphism, it follows that
    \[
    \Bbb{D}m=\mathrm{cl}(Skm) \cap \mathrm{int}(S^{-1}km).
    \]
    Hence, $\Bbb{D}m$ is a control set (see \cite{msm0}, Proposition 4.8). Since $k \in \Bbb{D}_0$, it follows that $k$ is self-accessible (see \cite{msm0}, Corollary 4.7), i.e., $k \in \mathrm{int}(S^{-1 }k)$. So,
    \[
    km \in \mathrm{int}(S^{-1}k)m=\mathrm{int}(S^{-1}km),
    \]
    i.e., $km \in (\Bbb{D}m)_0$, by Corollary 4.7 of \cite{msm0}, and
    \[
    (\Bbb{D}m)_0=Skm \cap S^{-1}km=(Sk \cap S^{-1}k)m=\Bbb{D}_0m.
    \]
    Furthermore, as $\pi(\Bbb{D}m)=\pi(\Bbb{D})$, for all $m \in M$, it follows, from Proposition \ref{prop:ccik}, that $ \Bbb{D}$ is $S$-invariant if and only if $\Bbb{D}m$ is $S$-invariant, which shows the item 1. For item 2, let $m\in M_0$ and $k \in \Bbb{D}_0$. By Proposition \ref{prop:kpreimcw},
    \[
    k\cdot m \in \Bbb{D}_0\cap \Bbb{D}_0m \subset \Bbb{D}\cap \Bbb{D}m
    \]
    and therefore $\Bbb{D}m=\Bbb{D}$.
\end{proof}

\begin{teorema}\label{teo:imainvDm}
    Let $w\in W$ and $\Bbb{D}\subset \pi^{-1}(\Bbb{C}(w))$ be a control set.
    \begin{enumerate}
	\item If $\Bbb{D}'\subset \pi^{-1}(\Bbb{C}(w))$ is a control set, then there exists $m \in M$ such that $\Bbb{D}'=\Bbb{D}m$.
	\item $\pi^{-1}(\Bbb{C}(w)_0)=\bigcup_{m \in M} \Bbb{D}_0m$, and
	\item If $\Bbb{D}$ is $S$-invariant, then $\pi^{-1}(\Bbb{C})=\bigcup_{m \in M} \Bbb{D}m$.
    \end{enumerate}
\end{teorema}

\begin{proof}
    Take $k' \in \Bbb{D}_0'$. By Proposition \ref{prop:pid0c0}, $\pi(\Bbb{D}_0)=\Bbb{C}(w)_0=\pi(\Bbb{D}_0')$. Thus, there exists $k \in \Bbb{D}_0$ such that $\pi(k)=\pi(k')$. Hence, there exists $m \in M$ such that $km=k'$, since $M$ is transitive on the fiber $\pi^{-1}(\pi(k))$ and thus $k'$ is in $\Bbb{D}'\cap \Bbb{D}m$ and therefore $\Bbb{D}'= \Bbb{D}m$, by Proposition \ref{prop:Dm}, which shows the item 1.
	
    For the second statement, let $k \in \pi^{-1}(\Bbb{C}(w)_0)$. By Propositions \ref{prop:kpreimcw} and \ref{prop:Dm}, and item 1, there exists $m \in M$ such that $k \in \Bbb{D}_0m \subset \pi^{-1} (\Bbb{C}(w)_0)$. Therefore, $\pi^{-1}(\Bbb{C}(w)_0)$ is contained in the union of sets $\Bbb{D}_0m$, with $m \in M$. For the other inclusion, just note that $\pi(\Bbb{D}m)=\pi(\Bbb{D})$, for all $m \in M$.
	
    Finally, for the third statement just use the second one, the Proposition \ref{prop:ccik}, the Corollary \ref{cor:ccswfinito} and note that the transitivity set is dense in the control set, $\pi^{- 1}(\Bbb{C}_0)$ is dense in $\pi^{-1}(\Bbb{C})$, since $\pi$ is an open map, and that $S$-invariant control sets are closed (see \cite{msm0}, Lemma 4.2 and Proposition 4.15).
\end{proof}

\subsection{Fixed points of $u$-type}\label{sec:pontfixu}

\textcolor{black}{In this subsection, we introduce the fixed points of $u$-type in the maximal compact subgroup $K$ which are the analog of the fixed points of $w$-type in the maximal flag manifold $K/M$, which are essential in the description of the control sets.} We denote by $b_0=1\cdot AN$ the origin of the homogeneous space $G/AN$. We have that $G/AN$ is identified with the subgroup $K$, by diffeomorphism
\begin{equation}\label{eq:difeoK}
    f: K \to G/AN, ~~ k \mapsto kb_0,
\end{equation}
which is equivariant with respect to the action of $G$ on $K$ and the usual action on $G/AN$. In fact,
\[
f(\delta(g,k))=f(\kappa(gk))=\kappa(gk)b_0=(gk)b_0=g(kb_0)=gf(k),
\]
for all $g \in G$ and $k \in K$, where $\kappa$ is the Iwasawa projection and $\delta$ the action of $G$ on $K$, defined in Subsection \ref{sec:conjuntocontrolek}.

\begin{lema}\label{lema:isotropiagb0}
    Take $g,g' \in G$.
    \begin{enumerate}
	\item The isotropy subgroup of $gb_0$ is $A(gB)N(gB)$.
	\item  $g'b_0=gb_0$ if and only if there are $n\in N(gB)$ and $a \in A$ such that $g'=nga$.
    \end{enumerate}
\end{lema}

\begin{proof}
    Since the isotropy subgroup of $b_0$ is $AN$ it follows that the isotropy subgroup of $gb_0$ is 
    \[
    gANg^{-1}=(gAg^{-1})(gNg^{-1})=A(gB)N(gB),
    \]
    which shows the item 1. For the second statement, we have that $g'b_0=gb_0$ if and only if $g^{-1}g'b_0=b_0$ which is equivalent to the existence of $a \in A$ and $n_0 \in N$ such that $g^{-1}g'=n_0a$, i.e.,
    \[
    g'=gn_0a=(gn_0g^{-1})ga=nga,
    \]
    with $n=gn_0g^{-1}\in gNg^{-1}=N(gB)$, which completes the proof.
\end{proof}

The characterization of the transitivity sets of the control sets of $S$ on $G/AN$ will be done through the action of the elements of $S$ that belong to some Weyl chamber in $G$. \textcolor{black}{The next proposition justifies the next definition.}

\begin{definicao}\label{def:pontofixou}
    For each $u \in U$ and $g \in G$, define the \emph{fixed point component of the type $u$} in $G/AN$ determined by $gB$ as
    \begin{equation}\label{eq:bu}
        b_u(gB) = u(gB)gb_0 = gu_*M_0(B)b_0
    \end{equation}
    where $u(B)=u_*(M_0A)(B)$. The elements of $b_u(gB)$ are called \emph{fixed points of type $u$}, determined by $gB$.
\end{definicao}

\begin{proposicao}\label{prop:fixtipou}
    \textcolor{black}{If $g \in G$, then the elements of $b_u(gB)$ are fixed points for any element of $\la(gB)$. Furthermore, if $h \in \la(gB)$ and $n \in N(gB)$, then $h^{-k}nh^k \to 1$ when $k \to \infty$.}
\end{proposicao}

\begin{proof}
    \textcolor{black}{If $h \in \la(gB)$, then there exists $h_0 \in \la_0$ such that $h = gh_0g^{-1}$. If $b \in b_u(gB)$, then there is $m_0 \in M_0(B)$ such that $b=gu_*m_0b_0$. Thus 
    \[
    hb = gh_0g^{-1}gu_*m_0b_0 = gu_*m_0(u_*m_0)^{-1}h_0u_*m_0b_0 = gu_*m_0b_0 = b
    \]
    since $(u_*m_0)^{-1}h_0u_*m_0 \in A$. If $n \in N(gB)$, then there exists $n_0 \in N$ such that $n = gn_0g^{-1}$. Thus
    \[
    h^{-k}nh^k = gh_0^{-k}g^{-1}gn_0g^{-1}gh_0^kg^{-1} = gh_0^{-k}n_0h_0^kg^{-1} \to gg^{-1} = 1
    \]
    since $h_0^{-k}n_0h_0^k \to 1$ when $k \to \infty$.
    }
\end{proof}

The component $b_1(gB)$ is denoted by $b(gB)$, where $1$ is the identity of $U$. Considering the identification $K \simeq G/AN$, given in (\ref{eq:difeoK}), the right action of an element $m \in M(B)$ on $gb_0 \in G/AN$, is given by
\[
(gb_0)\cdot m=gmb_0.
\]
If $b \in b_u(gB)$, then there is $m_0 \in M_0(B)$ such that $b=gu_*m_0b_0$. Hence
\begin{eqnarray}\label{eq:bu=bM0}
    b_{u}(gB)=gu_*M_0(B)b_0=(gu_*m_0b_0)\cdot M_0(B)=b\cdot M_0(B).
\end{eqnarray}
Thus, if $g_1$ and $g_2$ are elements of $G$ such that $b_u(g_1B)\cap b_u(g_2B) \neq \emptyset$ and $b \in b_u(g_1B)\cap b_u(g_2B) $, then 
\begin{eqnarray}\label{eq:bu1=bu2}
    b_u(g_1B)=b\cdot M_0(B)=b_u(g_2B).
\end{eqnarray}
\textcolor{black}{In particular, if $c \in C$ and $c(B)=c'(M_0A)(B)$, $c' \in M(B)$ and $g_1,g_2 \in G$ are such that $g_1b_0=g_2b_0$, then
\[
g_1c'b_0=(g_1b_0)\cdot c'=(g_2b_0)\cdot c'=g_2c'b_0 \in b_c(g_1B)\cap b_c(g_2B)
\]
and thus $b_c(g_1B) = b_c(g_2B)$. Therefore, Definition \ref{def:acaoC} is well defined.
}

\begin{definicao}\label{def:acaoC}
    Given $c \in C$, we define
    \[
    (gb_0)\cdot c=b_c(gB).
    \]
\end{definicao}

\begin{lema}\label{lema:isotropiafixu}
    If $g \in G$ and $u \in U$, then the stabilizer of $b_u(gB)$ is
    \[
    M_0(gB)A(gB)N_u(gB).
    \]
\end{lema}

\begin{proof}
    Take $g' \in G$. Using the notation $u(gB)=u_0(M_0A)(gB)$, $u_0\in (M_*A)(gB)$, the following conditions are equivalent:
    \begin{enumerate}[$(i)$]
	\item $g'b_u(gB)=b_u(gB)$,
	\item $g'u(gB)gb_0=u(gB)gb_0$,
	\item $g'u_0M_0(gB)gb_0=u_0M_0(gB)gb_0$,
	\item for all $m_0 \in M_0(gB)$, there exists $m_1 \in M_0(gB)$ such that 
	\[
	g'u_0m_0gb_0=u_0m_1gb_0,
	\]
	\item for all $m_0 \in M_0(gB)$, there exists $m_1 \in M_0(gB)$ such that 
	\[
	m_1^{-1}u_0^{-1}g'u_0m_0gb_0=gb_0,
	\]
	\item for all $m_0 \in M_0(gB)$, there exists $m_1 \in M_0(gB)$ such that 
	\[
	m_1^{-1}u_0^{-1}g'u_0m_0\in A(gB)N(gB),
	\]
	\item for all $m_0 \in M_0(gB)$, there exists $m_1 \in M_0(gB)$ such that 
	\[
	g'\in u_0m_1 A(gB)N(gB)m_0^{-1}u_0^{-1}=(u_0m_1m_0^{-1}u_0^{-1})A(gB)u_0N(gB)u_0^{-1},
	\]
	\item $g' \in M_0(gB)A(gB)N_u(gB)$.
    \end{enumerate}
    The equivalence between items \textcolor{black}{5} and \textcolor{black}{6} follows from Lemma \ref{lema:isotropiagb0} and in the equality in item \textcolor{black}{7} we use that $M_0(gB)$ centralizes $A(gB) $ and normalizes $N(gB)$, and that $(M_*A)(gB)$ normalizes $A(gB)$, and this completes the proof.
\end{proof}

\begin{lema}\label{lema:fixtipou}
    Let $g, g' \in G$. For all $u, u' \in U$ and $c \in C$, we have
    \begin{enumerate}
	\item $b_u(gg'B) = gb_u(g'B)$,
	\item $b_{u'}(u(gB)gB) = b_{uu'}(gB)$, and
	\item $b_u(gB)c=b_{uc}(gB)$.
    \end{enumerate}
    Furthermore, $b_u(g'B)=b_u(gB)$ if and only if $g'B=ngB$, for some $n \in N_u(gB)$.
\end{lema}

\begin{proof}
    By Definition \ref{def:pontofixou} and Proposition \ref{prop:Ucanonico},
    \[
    b_u(gg'B)
    =u(gg'B)gg'b_0
    =(gu(g'B)g^{-1})gg'b_0
    =gu(g'B)g'b_0
    =gb_u(g'B),
    \]
    which shows the item 1. Considering the notations $u(B)=u_*(M_0A)(B)$ and $u'(B)=u_*'(M_0A)(B)$, $u_*, u_*' \in \textcolor{black}{M_*}$, we have, from Lema \ref{lema:nu} and item 1, that
    \begin{eqnarray*}
	b_{u'}(u(gB)gB)
	&=&b_{u'}(gu_*B)
	=gu_*b_{u'}(B)
	=gu_*u'(B)b_0\\
	&=&g(u_*u'_*)M_0(B)b_0
	=g(uu')(B)b_0
	=(uu')(gB)gb_0\\
	&=&b_{uu'}(gB)
    \end{eqnarray*}
    and
    \[
    b_u(gB)c
    =(u(gB)gb_0)c
    =b_c(u(gB)gB)
    =b_{uc}(gB),
    \]
    which shows items 2 and 3.
	
    Furthermore, suppose that $b_u(g'B)=b_u(gB)$. By item 1,
    \[
    g^{-1}g'b_u(B)=g^{-1}b_u(g'B)=g^{-1}b_u(gB)=b_u(B).
    \]
    From Lemma \ref{lema:isotropiafixu}, $g^{-1}g' \in M_0(B)AN_u(B)=N_u(B)M_0(B)A=N_u(B)B$ and so
    \[
    g' \in gN_u(B)B=(gN_u(B)g^{-1})gB=N_u(gB)gB,
    \]
    by Lemma \ref{lema:nu}.	Taking $a \in B$ and $n \in N_u(gB)$ such that $g'=nga$ we have $g'B=ngaB=ngB$. Conversely, suppose $g'B=ngB$, with $n \in N_u(gB)$. By item 1 and Lemma \ref{lema:isotropiafixu},
    \[
    b_u(g'B)= b_u(ngB)=nb_u(gB)=b_u(gB),
    \]
    completing the proof.
\end{proof}

\begin{proposicao}\label{prop:acaocomppftu}
    Let $g_1, g_2 \in G$, $u \in U$ and $b_j \in b_u(g_jB)$, $j \in \{1,2\}$. If $g \in G$ is such that $gb_1=b_2$, then $gb_u(g_1B)=b_u(g_2B)$.
\end{proposicao}

\begin{proof}
    From Lemma \ref{lema:fixtipou}, we have that $gb_1 \in gb_u(g_1B)=b_u(gg_1B)$. So, by assumption, $b_2 \in b_u(g_2B)\cap b_u(gg_1B)$. Therefore, from Lemma \ref{lema:fixtipou} and the equality in (\ref{eq:bu1=bu2}), it follows that $gb_u(g_1B)=b_u(gg_1B)=b_u(g_2B)$.
\end{proof}

\subsection{Control sets and fixed points}\label{sec:contsettipou}

Considering the identification $K\simeq G/AN$, given in (\ref{eq:difeoK}), it follows that the control sets, of $S$, on $K$ are identified with the control sets, of $S$, on $G/AN$ in the sense that $\Bbb{D}\subset K$ is an (invariant) control set if and only if $f(\Bbb{D})\subset G /AN$ is an (invariant) control set. Next, $f$ is omitted and the same notation is considered for the control sets on $K$ and on $G/AN$.

\begin{proposicao}\label{prop:dc=dc*}
    Let $\Bbb{D}$ be a control set on $G/AN$ and $c \in C$. If $c(B)=c_*(M_0A)(B)$, $c_*\in \textcolor{black}{M(B)}$, then $\Bbb{D}_0c=\Bbb{D}_0c_*$ and $\Bbb{D}c=\Bbb{D}c_*$. In particular, $\Bbb{D}c$ is a control set and its transitivity set is $\Bbb{D}_0c$.
\end{proposicao}

\begin{proof}
    If $b=gb_0$, then
    \[
    b\cdot c_* = gc_*b_0 \in gc_*M_0(B)b_0=(gb_0)\cdot c=b\cdot c.
    \]
    Hence $\Bbb{D}_0c_*\subset \Bbb{D}_0c$ and $\Bbb{D}c_*\subset \Bbb{D}c$. 
	
    Conversely, let $b \in \Bbb{D}_0$. By Proposition \ref{prop:kpreimcw}, $b\cdot M_0(B) \subset \Bbb{D}_0$. So 
    \[
    b\cdot c = gc_*M_0(B)b_0=(b\cdot M_0(B)) c_*\subset \Bbb{D}_0c_*,
    \]
    which shows that $\Bbb{D}_0c\subset \Bbb{D}_0c_*$. Moreover, if $b'\in \Bbb{D}c$, then there are $g \in G$ and $m_0\in M_0(B)$ such that $gb_0 \in \Bbb{D}$ and $b '=gc_*m_0b_0$. Using the notation $\widetilde{m}_0=c_*m_0c_*^{-1}\in M_0(B)$ we have, from Proposition \ref{prop:Dm}, that $(gb_0)\cdot \widetilde {m}_0 \in \Bbb{D}$. So
    \[
    b'=gc_*m_0b_0=g\widetilde{m}_0c_*b_0=(g\widetilde{m}_0b_0)\cdot c_*=((gb_0)\cdot \widetilde{m}_0)\cdot c_* \in \Bbb{D}c_*,
    \]
    which shows that $\Bbb{D}c\subset\Bbb{D}c_*$ and completes the proof.
\end{proof}

Let $w \in W$ and $u \in U$ such that $w=\mathrm{Ad}(u)|_{\a}$ (see definition on page \pageref{def:w=Adu}). If $\pi: G/AN \to \F$ denotes the canonical projection of $G/AN$ onto the maximal flag manifold $\F$, then, from equality in (\ref{eq:bu}), we have
\begin{equation}\label{eq:projfixu0}
    \pi(b_{u}(gB))=gu_*\p(B)=gw(B)\p(B)=w(gB)\p(gB)=\p_w(gB).
\end{equation}

\begin{definicao}\label{def:B(S)}
    The \emph{set of cosets determined by $S$, $\mathcal{B}(S)$}, is defined by
    \[
    \mathcal{B}(S)=\{gB:\la(gB) \cap S \neq \emptyset\},
    \]
    where $\la(gB)=g\la_0g^{-1}$.
\end{definicao}

\begin{lema}\label{lema:classelateralS}
    If $gB \in \mathcal{B}(S)$ and $c \in C$, then $c(gB)gB \in \mathcal{B}(S)$.
\end{lema}

\begin{proof}
    Writing $c(B)=c_*(M_0A)(B)$, with \textcolor{black}{$c_*\in M(B)$}, we have from Lemma \ref{lema:nu} that $c(gB)gB=gc_ *B$. Moreover, as \textcolor{black}{$M(B)$} centralizes $\la_0$, we have
    \[
    \la(gc_*B)=gc_*\la_0(gc_*)^{-1}=gc_*\la_0c_*^{-1}g^{-1}=g\la_0g^{-1}=\la(gB),
    \]
    which shows that $\la(gc_*B)$ intersects $S$. So $c(gB)gB \in \mathcal{B}(S)$.
\end{proof}

\begin{lema}\label{lema:cpftu-in-d0}
    Let $\Bbb{D}$ be a control set on $G/AN$. If $u \in U$ and $g \in G$ are such that $b_u(gB)\cap \Bbb{D}_0\neq \emptyset$, then $b_u(gB) \subset \Bbb{D}_0$.
\end{lema}

\begin{proof}
    Take $b\in b_u(gB)\cap \Bbb{D}_0$. By equality in (\ref{eq:bu=bM0}) and Proposition \ref{prop:kpreimcw},
    \[
    b_u(gB)=b\cdot M_0(B)\subset \Bbb{D}_0,
    \]
    which completes the proof.
\end{proof}

\begin{proposicao}\label{prop:d0-cpftu}
    Let $w\in W$, $\Bbb{D}\subset \pi^{-1}(\Bbb{C}(w))$ a control set and $u \in U$ such that $w=\mathrm{Ad}(u)|_{\a}$. If $b \in \Bbb{D}_0$, then there exists $gB$ in $\mathcal{B}(S)$ such that $b\in b_u(gB)\subset \Bbb{D}_0$.
\end{proposicao}

\begin{proof}
    By the Proposition \ref{prop:pid0c0}, $\pi(\Bbb{D}_0)= \Bbb{C}(w)_0$. As $b \in \Bbb{D}_0$ follows, from Proposition \ref{prop:controleweyl}, that there exists $g'B\in \mathcal{B}(S)$ such that $\pi(b)= \p_w(g'B)$. By equality in (\ref{eq:projfixu0}), $\pi(b_{u}(g'B))=\p_w(g'B)$. Since $b_{u}(g'B)=g'u_*M_0(B)b_0$, by (\ref{eq:bu}), where $u(B)=u_*(M_0A)(B)$, and $M(B)$ is transitive on the fiber, there exists $c_* \in M(B)$ such that
    \[
    b=(g'u_*b_0)\cdot c_*=g'u_*c_*b_0.
    \]
    Let $c \in C$ be such that $c(B)=c_*(M_0A)(B)$, by Proposition \ref{prop:Ucanonico}. By Lemma \ref{lema:fixtipou},
    \[
    b \in g'u_*c_*M_0(B)b_0=b_{uc}(g'B)=b_u(\widetilde{c}(g'B)g'B),
    \]
    where $\widetilde{c}=ucu^{-1} \in C$. Since $g'B \in \mathcal{B}(S)$ it follows, from Lemma \ref{lema:classelateralS}, that $\widetilde{c}(g'B)g'B \in \mathcal{B}(S)$. By Lemma \ref{lema:nu}, $\widetilde{c}(g'B)g'B=g'\widetilde{c}_*B$. Therefore taking $g=g'\widetilde{c}_*$, by Lemma \ref{lema:cpftu-in-d0}, we have that
    \[
    b \in b_u(gB)\subset \Bbb{D}_0,
    \]
    since $b \in \Bbb{D}_0$.
\end{proof}

\begin{proposicao}\label{prop:cpftu-d0}
    If $u \in U$ and $gB \in \mathcal{B}(S)$, then there exists a control set $\Bbb{D}$, on $G/AN$, such that $b_u(gB) \subset \Bbb{D}_0$. Moreover, if $u \in C$, then $\Bbb{D}$ is $S$-invariant.
\end{proposicao}

\begin{proof}
    Let $w=\mathrm{Ad}(u)|_{\a} \in W$. From (\ref{eq:bu}) and Proposition \ref{prop:controleweyl},
    \[
    \pi(b_u(gB))=\p_w(gB)\in \Bbb{C}(w)_0.
    \]
    Writing $u(B)=u_*(M_0A)(B)$ and taking $b=gu_*b_0\in b_u(gB)\cap \pi^{-1}(\p_w(gB))$ there exists, by Proposition \ref{prop:kpreimcw}, a control set $\Bbb{D}\subset G/AN$ such that $b \in \Bbb{D}_0$. Therefore, by Lemma \ref{lema:cpftu-in-d0}, $b_u(gB) \subset \Bbb{D}_0$. Moreover, if $u \in C$, then $w=1$ and therefore, by Proposition \ref{prop:kpreimcw}, $\Bbb{D}$ is $S$-invariant.
\end{proof}

\begin{lema}\label{lema:controleul1}
    Let $g, g' \in G$ be such that $gB, g'B \in \mathcal{B}(S)$.
    \begin{enumerate}
        \item\label{lema:controleul11} \textcolor{black}{If $b(gB)$ and $b(g'B)$ are contained in the transitivity set of the same invariant control set}, then for all $u \in U$ we have that $b_u(gB)$ and $b_u(g'B)$ are contained in the transitivity set of the same control set.
        \item\label{lema:controleul12}  \textcolor{black}{If $b_u(gB)$ and $b_u(g'B)$ are contained in the transitivity set of the same control set,} for some $u \in U$, then $b(gB)$ and $b(g'B)$ are contained in the transitivity set of the same invariant control set.
    \end{enumerate} 
\end{lema}

\begin{proof}
    Let $u \in U$. By Proposition \ref{prop:cpftu-d0}, there are control sets $\Bbb{D}$ and $\Bbb{D}'$, on $G/AN$, such that $b_u(gB)\subset \Bbb{D}_0$ and $b_u(g'B)\subset \Bbb{D}'_0$. \textcolor{black}{If $b_1 \in b(gB)$ and $b_2 \in b(g'B)$, then by hypothesis and definition of transitivity set, there exists $\widetilde{g} \in S$ such that $\widetilde{g}b_1=b_2$. By Proposition \ref{prop:acaocomppftu}, $\widetilde{g}b(gB)=b(g'B)$. From Lemma \ref{lema:fixtipou}, we have that
    \[
    b(gB)=(\widetilde{g})^{-1} b(g'B)= b((\widetilde{g})^{-1}g'B).
    \]} 
    By Lemma \ref{lema:fixtipou}, there exists $n \in N(gB)$ such that $\textcolor{black}{(\widetilde{g})^{-1}}g'B=ngB$. Hence, $b_u(\textcolor{black}{(\widetilde{g})^{-1}}g'B) = nb_u(gB)$, by the Lemma \ref{lema:fixtipou}. Using the notation $u(B)=u_*(M_0A)(B)$ we have, by (\ref{eq:bu}), that
    \[
    \textcolor{black}{(\widetilde{g})^{-1}}g'u_*M_0(B)b_0=ngu_*M_0(B)b_0.
    \]
    Hence, there is $m_0\in M_0(B)$ such that $\textcolor{black}{(\widetilde{g})^{-1}}g'u_*b_0=ngu_*m_0b_0$. Let $h \in \la(gB) \cap S$. By Proposition \ref{prop:fixtipou}, $gu_*m_0b_0$ is a fixed point of $h^k$, for all $k \in \N$. Thus, if $k \rightarrow \infty$, then, \textcolor{black}{by Proposition \ref{prop:fixtipou},}
    \begin{equation*}
        h^{-k}\textcolor{black}{(\widetilde{g})^{-1}}g'u_*b_0 = h^{-k}ngu_*m_0b_0 = h^{-k}nh^kgu_*m_0b_0 \rightarrow gu_*m_0b_0\in b_u(gB)\subset \Bbb{D}_0.
    \end{equation*}
    Since $S$ is an open semigroup, we have that $\Bbb{D}_0$ is an open set and, therefore, there is $k_0$ such that $h^{-k_0}\textcolor{black}{(\widetilde{g})^{-1}}g'u_*b_0 \in\Bbb{D}_0$. This shows that $\Bbb{D} \leq \Bbb{D}'$, since $g'u_*b_0 \in b_u(g'B)\subset\Bbb{D}'_0$ and \textcolor{black}{$\widetilde{g}h^{k_0} \in S$}, and reversing the roles of $gB$ and $g'B$ we get equality, which shows item 1.
	
    For item 2, let $\Bbb{D}$ and $\Bbb{D}'$ be invariant control sets such that $b(gB)$ and $b(g'B)$ are contained in $\Bbb{D}_0$ and $\Bbb{D}'_0$, respectively, by Proposition \ref{prop:cpftu-d0}. \textcolor{black}{If $b_1 \in b_u(g'B)$ and $b_2 \in b_u(gB)$, then by hypothesis and definition of transitivity set, there exists $\widetilde{g} \in S$ such that $\widetilde{g}b_1=b_2$. From Proposition \ref{prop:acaocomppftu} and  Lemma \ref{lema:fixtipou}, we have that
    \[
    b_u(gB)=\widetilde{g} b_u(g'B)= b_u(\widetilde{g}g'B).
    \]} 
    From Lemma \ref{lema:fixtipou}, $\textcolor{black}{\widetilde{g}}g'=ngm_0$, with $n\in N_u(gB)$ and $m_0\in (M_0A)(B)$. By Lemma \ref{lema:nu}, there exists $n_0\in N$ such that $n=gu_*n_0u_*^{-1}g^{-1}$, where $u(B)=u_*(M_0A)( B)$, with \textcolor{black}{$u_*\in M_*$}. Hence,
    \[
    \textcolor{black}{\widetilde{g}}g'= (gu_*n_0u_*^{-1}g^{-1})gm_0=gu_*n_0u_*^{-1}m_0\in gu_*Nu_*^{-1}B
    \] 
    so $g^{-1}\textcolor{black}{\widetilde{g}}g' \in u_*Nu_*^{-1}B=N_u(B)B$. By Proposition \ref{prop:proposicao2},
    \begin{eqnarray*}
        g^{-1}\textcolor{black}{\widetilde{g}}g'\in (N_u(B)\cap N^{-})(N_u(B)\cap N)B,
    \end{eqnarray*}
    i.e., $g^{-1}\textcolor{black}{\widetilde{g}}g'=n^-n'm$, with $n^- \in N_u(B)\cap N^{-}$, $n' \in N_u(B) \cap N$ and $m \in (M_0A)(B)$, which implies that $g^{-1}\textcolor{black}{\widetilde{g}}g'b_0=n^-mb_0$. Take $h \in \la(gB)\cap S$. Thus $g^{-1}hg \in \la_0$ and, \textcolor{black}{by Proposition \ref{prop:fixtipou},} $mb_0$ is a fixed point of $(g^{-1}hg)^{k}$, for all $k \in \N $. So if $k \rightarrow \infty$, then
    \begin{eqnarray*}
        h^{k}\textcolor{black}{\widetilde{g}}g'b_0
	&=& h^{k}gg^{-1}\textcolor{black}{\widetilde{g}}g'b_0
	=h^{k}gn^-mb_0\\
	&=&g(g^{-1}hg)^{k}n^-(g^{-1}hg)^{-k}mb_0 \rightarrow gmb_0\in b(gB)\subset \Bbb{D}_0.
    \end{eqnarray*}
    Since $S$ is an open semigroup, $\Bbb{D}_0$ is an open set and, therefore, there is $k_0$ such that $h^{k_0}\textcolor{black}{\widetilde{g}}g'b_0 \in\Bbb{ D}_0$. This shows that $\Bbb{D}' \leq \Bbb{D}$, since $g'b_0 \in b(g'B)\subset \Bbb{D}_0'$ \textcolor{black}{and $h^{k_0}\widetilde{g} \in S$}, and reversing the roles of $gB$ and $g'B$ we get equality, which concludes the proof.
\end{proof}

\textcolor{black}{In the following theorem, the first main result of the paper, we show that the control sets of $S$ on the maximal compact subgroup $K$ are parameterized by the elements of the group $U$. This is the precise analog of what happens in the maximal flag manifold $K/M$, where the control sets of $S$ are parameterized by the elements of the Weyl group $W$.}

\begin{teorema}\label{teo:Du}
\textcolor{black}{
    Let $\Bbb{D}$ be an invariant control set on $G/AN$. Given $u \in U$, there exists a control set $\Bbb{D}(u)$ with the following property: If $gB \in \mathcal{B}(S)$, then $b(gB)\subset \Bbb{D}_0$ if and only if $b_u(gB)\subset \Bbb{D}(u)_0$. Moreover, if $c \in C$, then
    \begin{enumerate}
	\item $\Bbb{D}(uc)=\Bbb{D}(u)c$ and $\Bbb{D}(cu)=\Bbb{D}(u)u^{-1}cu$.
	\item $\Bbb{D}(c)$ is an invariant control set and $\Bbb{D}(1)=\Bbb{D}$.
	\item $\Bbb{D}c=\Bbb{D}$ if and only if $\Bbb{D}(cu)=\Bbb{D}(u)$.
        \item $\pi(\textcolor{black}{\Bbb{D}(u)_0})=\Bbb{C}(w)_0$, if $w=\mathrm{Ad}(u)|_{\a}$.
    \end{enumerate}
}
\end{teorema}
\begin{proof}
\textcolor{black}{
    By the Proposition \ref{prop:d0-cpftu}, there exists $g'B$ in $\mathcal{B}(S)$ such that $b(g'B)\subset \Bbb{D}_0$. By the Proposition \ref{prop:cpftu-d0}, there exists a control set $\Bbb{D}(u)$ such that $b_u(g'B) \subset \Bbb{D}(u)_0$. Therefore, if $gB \in \mathcal{B}(S)$, then, from Lemma \ref{lema:controleul1}, $b(gB) \subset \Bbb{D}_0$ if and only if $b_u(gB) \subset \Bbb{D}(u)_0$.
}

\textcolor{black}{
    Moreover, if $gB\in \mathcal{B}(S)$ is such that $b(gB)\subset \Bbb{D}_0$, then, from Lemma \ref{lema:fixtipou},
    \[
    b_{uc}(gB)=b_{u}(gB)c\subset \Bbb{D}(uc) \cap \Bbb{D}(u)c.
    \]
    Therefore $\Bbb{D}(uc) = \Bbb{D}(u)c$, which shows the first equality in 1. The second equality is a consequence of the first one and the fact that $u^{-1}cu \in C$. The first statement in 2 follows from Proposition \ref{prop:cpftu-d0} and the second one follows from the definition of $\Bbb{D}(1)$ and the fact that $b_1(gB)=b(gB)$. 
}

\textcolor{black}{
    To show item 3, let $gB \in \mathcal{B}(S)$ such that $b(gB) \subset \Bbb{D}_0$. From Lemma \ref{lema:classelateralS}, $c(gB)gB \in \mathcal{B}(S)$. From Lemma \ref{lema:fixtipou},
    \[
    b(c(gB)gB)=b(gB)c\subset \Bbb{D}_0c \text{ and } b_u(c(gB)gB)=b_{cu}(gB) \subset \Bbb{D}(cu)_0.
    \]
    Hence, if $\Bbb{D}c=\Bbb{D}$, then 
    $b_u(c(gB)gB) \subset \Bbb{D}(cu) \cap \Bbb{D}(u)$, which shows that $\Bbb{D}(cu) = \Bbb{D}(u)$. Conversely, if $\Bbb{D}(cu)=\Bbb{D}(u)$, then 
    $b(c(gB)gB) \subset \Bbb{D}c \cap \Bbb{D}$, which shows that $\Bbb{D}c = \Bbb{D}$.
}

\textcolor{black}{
    For item 4, since $\Bbb{D}$ is an invariant control set there exists, by Proposition \ref{prop:d0-cpftu}, $g \in G$ such that
    \[
    gB \in \mathcal{B}(S)~\text{ and }~gb_0\in b(gB)\subset \Bbb{D}_0
    \]
    and thus $b_u(gB) \subset \Bbb{D}(u)_0$. It follows, from the equality in (\ref{eq:projfixu0}), that $\pi(b_u(gB))=\p_w(gB)$. As $\p_w(gB) \in \Bbb{C}(w)_0$, by Proposition \ref{prop:controleweyl}, it follows that $\pi(\Bbb{D}(u)_0)=\Bbb{C}(w)_0$, by Proposition \ref{prop:pid0c0}, which completes the proof.
}
\end{proof}

\subsection{Characterization of control sets}\label{sec:desccontrolset}
By choosing $\la_0$, at the beginning of Section \ref{sec:mainresults}, we have $\la_0 \cap S \neq \emptyset$. So by Definition \ref{def:B(S)}, $B \in \mathcal{B}(S)$. Let $\Bbb{D}$ be the invariant control set such that
\[
b(B)=M_0(B)b_0 \subset \Bbb{D}_0,
\]
by Proposition \ref{prop:cpftu-d0}. \emph{The set of cosets determined by $S$, associated with \textcolor{black}{$\Bbb{D}$}}, \textcolor{black}{$\mathcal{B}_0(S)$}, is defined by
\begin{equation*}\label{not:BSc}
    \textcolor{black}{\mathcal{B}_0(S)}=\{gB: gB \in \mathcal{B}(S) \text{ and } b(gB) \subset \textcolor{black}{\Bbb{D}_0}\}.
\end{equation*}

\begin{definicao}\label{def:fix-atr}
    Take $u \in U$. The set of \emph{fixed points of $S$ of type $u$} is given by
    \textcolor{black}{
\begin{equation}\label{eq:fixSuc}
    \mathrm{fix}_u(S)=\bigcup_{gB \in \mathcal{B}_0(S)}b_{u}(gB).
\end{equation}}
\end{definicao}

\textcolor{black}{In the following theorem, the second main result of the paper, we show that the transitive sets of the control sets of $S$ on the maximal compact subgroup $K$ are characterized by the fixed points of $u$-type. This is the precise analog of what happens in the maximal flag manifold $K/M$, where the transitive sets of the control sets of $S$ are characterized by the fixed points of $w$-type.}

\begin{teorema}\label{teo:controleweylk}
    If $u \in U$, then \textcolor{black}{$\Bbb{D}(u)_0=\mathrm{fix}_{u}(S)$}. Furthermore, if $\Bbb{D}' \subset G/AN$ is a control set, then there exist $u \in U$ such that $\Bbb{D}'=\textcolor{black}{\Bbb {D}(u)}$.
\end{teorema}

\begin{proof}
    \textcolor{black}{The inclusion $\mathrm{fix}_u(S) \subset \Bbb{D}(u)_0$ follows from (\ref{eq:fixSuc}) and Theorem \ref{teo:Du}}. To show equality, take $b \in \textcolor{black}{\Bbb{D}(u)_0}$. \textcolor{black}{We have from Theorem \ref{teo:Du}, item 4,} and Proposition \ref{prop:d0-cpftu}, that there is $gB$ in $\mathcal{B}(S)$ such that $b\in b_{u}(gB)\subset \textcolor{black}{\Bbb{D}(u)_0}$. \textcolor{black}{By Theorem \ref{teo:Du}, $b(gB) \subset \Bbb{D}_0$, and from (\ref{eq:fixSuc}), $b\in b_{u}(gB)\subset \mathrm{fix}_{u}(S)$. Therefore, $\Bbb{D}(u)_0=\mathrm{fix}_u(S)$.}
	
    For the second statement, let $w \in W$ and $u \in U$ be such that $\pi(\Bbb{D}'_0)=\Bbb{C}(w)_0$, by Proposition \ref{prop:pid0c0}, and $w=\mathrm{Ad}(u)|_{\a}$. By Theorems \ref{teo:imainvDm} and \ref{teo:Du}, \textcolor{black}{items 1 and 4}, and Proposition  \ref{prop:dc=dc*}  we have that there exists $c \in C$ such that
    \[
    \Bbb{D}'=\textcolor{black}{\Bbb{D}(u)c=\Bbb{D}(uc)},
    \]
    which completes the proof.
\end{proof}

Given $g \in G$ we have
\[
\frac{U(gB)}{C(gB)}\simeq W(gB)\simeq \Ad((M_*A)(gB))|_{\a(gB)}.
\]
Hence, for every $w \in W$, there exists $u \in U$ such that $u(gB)=u_*(M_0A)(gB)$, with $u_*\in (M_*A)(gB)$, and
\[
u(gB)C(gB)=w(gB)=\Ad(u_*)|_{\a(gB)}.
\]
Then, using the notation introduced on the page \pageref{def:w=Adu}, we have that
\[
uC=w=\Ad(u)|_{\a}.
\]
The group $W$ acts on $C$ by conjugation. Indeed, given $w\in W$ and $c\in C$, there exists $u \in U$ such that $w=uC$. We define
\begin{equation}\label{eq:acaowc}
    wcw^{-1}=ucu^{-1}.
\end{equation}
Since $C$ is abelian, the equality in (\ref{eq:acaowc}) is well defined, i.e., it does not depend on the representative $u \in U$, it only depends on the coset $w=uC$.

Let $w \in W$. By Propositions \ref{prop:kpreimcw} and \ref{prop:pid0c0} there exists a control set $\widetilde{\Bbb{D}}$ such that $\pi(\widetilde{\Bbb{D}}_0)=\Bbb{C}(w )_0$. Let us consider the subset $C(S,w)$ of $C$ defined by
\begin{equation}\label{eq:msw}
    C(S,w)=\{c \in C : \widetilde{\Bbb{D}}c=\widetilde{\Bbb{D}}\}.
\end{equation}
It is clear that $C(S,w)$ is a subgroup of $C$. The subgroup $C(S,1)$ is simply denoted by $C(S)$. If $\Bbb{D}'$ is another control set such that $\pi(\Bbb{D}'_0)=\Bbb{C}(w)_0$, then, by Propositions \ref{prop:Dm}, \ref{prop:Ucanonico} and \ref{prop:dc=dc*}, there exists $c \in C$ such that $\Bbb{D}'=\widetilde{\Bbb{D}}c$. Hence, using the fact that $C$ is abelian, it is straightforward to show that the subgroup $C'(S,w)$ defined by $\Bbb{D}'$ is equal to $C(S,w)$.

\begin{proposicao}\label{msw=ms1}
    If $w \in W$, then $C(S,w)=w^{-1}C(S)w$.
\end{proposicao}

\begin{proof}
    \textcolor{black}{Let $c \in C$. Writing $w=uC$, with $u \in U$, we have from Theorem \ref{teo:Du}, items 1 and 3, that $\Bbb{D}c=\Bbb{D}$ if and only if $\Bbb{D}(u)u^{-1}cu=\Bbb{D}(cu)=\Bbb{D}(u)$, which completes the proof.}
\end{proof}

\textcolor{black}{The following theorem, the third main result of the paper, shows that the number of control sets of $S$ on the maximal compact subgroup $K$ is given by the product of the number of $S$-invariant control sets on $K$ times the number of control sets of $S$ on the maximal flag manifold $K/M$.}

\begin{teorema}\label{teo:nocontrolsetk}
    If $w \in W$, then there is a bijection between the control sets on $\pi^{-1}(\Bbb{C}(w))$ and the set of cosets $C(S,w) \setminus C$. Furthermore, the number of control sets of $S$ on $K\simeq G/AN$ is the product of the cardinalities of $C(S)\setminus C$ and $W(S)\setminus W$.
\end{teorema}
\begin{proof}
    Let $u \in U$ be such that $w=\mathrm{Ad}(u)|_{\a}$. By \textcolor{black}{Theorem \ref{teo:Du}, item 4,} Proposition \ref{prop:dc=dc*}, and Theorem \ref{teo:imainvDm}, $\pi^{-1}(\Bbb{C}(w)_0) =\bigcup_{c \in C}\textcolor{black}{\Bbb{D}(u)}_0c$. As $\textcolor{black}{\Bbb{D}(u)c=\Bbb{D}(u)c'}$ if and only if $C(S,w)c=C(S,w)c' $ we have that
    \[
    \pi^{-1}(\Bbb{C}(w)) \to C(S,w)\setminus C, ~~ \textcolor{black}{\Bbb{D}(u)}c \mapsto C(S,w)c
    \]
    is a well defined and injective map. This map is also surjective and this shows the first statement. For the second part, we combine the first one with the fact that subgroups $C(S,w)$, with different $w$, have the same cardinality of $C(S)$, by Proposition \ref{msw=ms1}, and that the number of control sets on the maximal flag manifold is the cardinality of $W(S)\setminus W$.
\end{proof}

Let $\Theta = \Theta(S)$ be the flag type of $S$. Let $\pi_{\Theta}:G/AN \to \F_{\Theta}$ be the canonical projection and $\Bbb{C}_{\Theta}$ the $S$-invariant control set on $ \F_{\Theta}$. If $K_{\Theta}$ denotes the centralizer of $\a_{\T}$ in $K$, then the fibers of $\pi_{\Theta}$ are diffeomorphic to $P_{\Theta}/AN\simeq K_{\Theta}$ and $\pi_{\Theta}^{-1}(\Bbb{C}_{\Theta})$ is diffeomorphic to $\Bbb{C}_{\Theta} \times K_{ \Theta}$. The right action of $M$ on $\Bbb{C}_{\Theta} \times K_{ \Theta}$ is given by $(b,k)\cdot m =(b,km)$, with $m \in M$ and $(b,k) \in \Bbb{C}_{\Theta} \times K_{\Theta}$. Furthermore, if $S$ is connected and $\Bbb{D} \subset G/AN$ is a $S$-invariant control set, then $\Bbb{D}\simeq \Bbb{C}_{\Theta} \times K_{\Theta}^0$, where $K_{\Theta}^0$ is the identity component of $K_{\Theta}$ (see \cite{sms}, Corollary 3.4). Therefore, denoting by $[m]$ the class of an element $m \in M$, in the group $C$, we have
\[
C(S)=\{[m]:\Bbb{D}m=\Bbb{D}\}=\left\{[m]:m \in K_{\T}^0\right\}=\frac{M\cap K_{\T}^0}{M_0}.
\]
As $W(S)=W_{\T}$ we have the following result.

\begin{teorema}
    If $S$ is connected, then the group $C(S)$ is determined by $W(S)$. In particular, if $W(S)=\{1\}$, then $C(S)=\{1\}$ and if $W(S)=W$, then $C(S)=C$ .
\end{teorema}

\subsection{Examples}\label{sec:exemplos}

\begin{exemplo}
    Let $G=\mathrm{Sl}(2,\R)$ be the group of ~$2\times 2$ matrices with entries in $\R$ and determinant 1. If $A$ and $N$ denote the subgroups of $G$ of the diagonal matrices with positive entries and the upper triangular matrices with ones on the diagonal, respectively, then $G$ decomposes as $G=\mathrm{SO(2)}AN$. The centralizer and normalizer of $A$ in $\mathrm{SO(2)}$ denoted, respectively, by $M$ and $M_*$ are given by
    \[
    M=\left\{
    \left(
    \begin{array}{cc}
        1&0\\
	0&1
    \end{array}
    \right),
    \left(
    \begin{array}{cc}
        -1&0\\
	0&-1
    \end{array}
    \right)
    \right\}
    \]
    and
    \[
    M_*=\left\{
    \left(
    \begin{array}{cc}
        1&0\\
	0&1
    \end{array}
    \right),
    \left(
    \begin{array}{cc}
        0&-1\\
        1&0
    \end{array}
    \right),
    \left(
    \begin{array}{cc}
        -1&0\\
	0&-1
    \end{array}
    \right),
    \left(
    \begin{array}{cc}
        0&1\\
	-1&0
    \end{array}
    \right)
    \right\}.
    \]
    The Weyl group is
    \[
    W=\left\{
    \left(
    \begin{array}{cc}
        1&0\\
	0&1
    \end{array}
    \right),
    \left(
    \begin{array}{cc}
        0&1\\
	1&0
    \end{array}
    \right)
    \right\}.
    \]
	
    The group $G$ acts on the projective space $\Bbb{P}^1$ by $g[x]=[gx]$, where $[x]$ is the subspace generated by the non-zero vector $x \in \R^ 2$. Identifying the subgroup $\mathrm{SO(2)}$ with the space of rays with origin at $(0,0)$, we have that the action of $G$ on $\mathrm{SO(2)}$, induced by the Iwasawa decomposition, is given by $g\langle x\rangle=\langle gx\rangle$, where $\langle x\rangle$ is the ray with origin at $(0,0)$ and containing the non-zero vector $x$.
	
    Given $H=\mathrm{diag}(a,-a)$, with $a>0$, we have that
    \[
    h^t=\mathrm{exp}(tH)=\mathrm{diag}(e^{ta},e^{-ta}).
    \]
    So
    \begin{itemize}
        \item In $\Bbb{P}^1$, $h^t$ has two fixed points, namely,
	\[
	\p=[(1,0)] ~\textrm{ and }~ w\p=[(0,1)],
	\]
	where $w=\left(
	\begin{array}{cc}
		0&1\\
		1&0
	\end{array}
        \right) \in W$. Moreover, $\p$ is the only attractor and $w\p$ is the only repeller of $h^t$ in $\Bbb{P}^1$. The stable manifold of $\p$ is open and dense in $\Bbb{P}^1$.
	\item In $\mathrm{SO}(2)$, $h^t$ has four fixed points, namely,
	\[
	b_0=\langle (1,0)\rangle, ~ u_1b_0=\langle (0,1)\rangle, ~ cb_0=\langle (-1,0)\rangle ~\textrm{ and }~ u_2b_0=\langle           (0,-1)\rangle,
	\]
	where $c=-\mathrm{Id} \in M$ and 
	$u_1=\left(
	\begin{array}{cc}
		0&-1\\
		1&0
	\end{array}
	\right)$,
	$u_2=\left(
	\begin{array}{cc}
		0&1\\
		-1&0
	\end{array}
        \right)\in M_*$. Moreover, $b_0$ and $cb_0$ are the attractors and $u_1b_0$ and $u_2b_0$ are the repellers of $h^t$ in $\mathrm{SO}(2)$. The union of stable manifolds of attractors is open and dense in $\mathrm{SO}(2)$.
    \end{itemize} 
    
    Let $S = \mathrm{Sl}^+(2,\R)$ be the semigroup of matrices, in $G$, with positive entries. We have that $S$ is an open semigroup of $G$. Given $x=(x_1,x_2)$ and $y=(y_1,y_2)$ in $\R^2$, with $x_i,y_i>0$, let
    \[
    g=\sqrt{\frac{x_1x_2}{3y_1y_2}}
    \left(
    \begin{array}{cc}
	2y_1/x_1 & y_1/x_2 \\
	y_2/x_1 & 2y_2/x_2 \\
    \end{array}
    \right).
    \]
    We have $g\in S$, $g[x]=[y]$ and $g\langle x \rangle =\langle y \rangle$. Therefore, $S$ is transitive on $\mathrm{int}\Bbb{C}$ and $\mathrm{int}\Bbb{D}$, where
    \[
    \Bbb{C}=\{[(x_1,x_2)]:x_i\geq 0, ~\text{for all}~ i\}  \subset \Bbb{P}^1
    \]
    and
    \[
    \Bbb{D}=\{\langle(x_1,x_2)\rangle:x_i\geq 0, ~\text{for all}~ i\} \subset \mathrm{SO}(2).
    \]
    Moreover taking 
    \[
    g'=
    \left(
    \begin{array}{cc}
	2 & 1 \\
	1 & 1 \\
    \end{array}
    \right) \in S
    \]
    we have that $g'(1,0)=(2,1)$ and $g'(0,1)=(1,1)$, i.e., $g'[(1,0)]$ and  $g'[(0,1)]$ are in $\mathrm{int}\Bbb{C}$ and $g'\langle (1,0)\rangle$ and  $g'\langle(0,1)\rangle$ are in $\mathrm{int}\Bbb{D}$. Hence $\mathrm{cl}(S[x])=\Bbb{C}$ for all $[x] \in \Bbb{C}$, and $\mathrm{cl}(S\langle x\rangle)=\Bbb{D}$ for all $\langle x \rangle \in \Bbb{D}$ and this implies that $\Bbb{C}$ and $\Bbb{D}$ are invariant control sets of $S$ on $\Bbb{P}^1$ and $\mathrm{SO}(2)$, respectively. Similarly, the set $\Bbb{D}c$, which is given by
    \[
    \Bbb{D}c=\{\langle(x_1,x_2)\rangle:x_i\leq 0, ~\text{for all}~ i\} \subset \mathrm{SO}(2)
    \]
    is an invariant control set for $S$ on $\mathrm{SO}(2)$.
	
    Given $x=(x_1,x_2)$ and $y=(y_1,y_2)$ in $\R^2$, with $x_1,y_1<0$ and $x_2,y_2>0$, let
    \[
    g=\sqrt{\frac{x_1x_2}{3y_1y_2}}
    \left(
    \begin{array}{cc}
	2y_1/x_1 & -y_1/x_2 \\
	-y_2/x_1 & 2y_2/x_2 \\
    \end{array}
    \right).
    \]
    We have $g\in S$, $g[x]=[y]$ and $g\langle x \rangle =\langle y \rangle$. Therefore, $S$ is transitive on $\Bbb{C}'$ and $\Bbb{D}'$, where
    \[
    \Bbb{C}'=\{[(x_1,x_2)]:x_1< 0 ~\textrm{and}~ x_2>0\}  \subset \Bbb{P}^1
    \]
    and
    \[
    \Bbb{D}'=\{\langle(x_1,x_2)\rangle:x_1<0 ~\text{and}~ x_2>0\} \subset \mathrm{SO}(2).
    \] 
    Hence, $\Bbb{C}'\subset \mathrm{cl}(S[x])$, for all $[x] \in \Bbb{C}'$ and $\Bbb{D}'\subset \mathrm{cl}(S\langle x\rangle)$, for all $\langle x \rangle \in \Bbb{D}'$ and thus $\Bbb{C}'$ and $\Bbb{D}'$ are control sets of $S$ on $\Bbb{P}^1$ and $\mathrm{SO}(2)$, respectively. Similarly, the set $\Bbb{D}'c$, which is given by
    \[
    \Bbb{D}'c=\{\langle(x_1,x_2)\rangle:x_1>0 ~\text{and}~ x_2<0\} \subset \mathrm{SO}(2)
    \]
    is a control set of $S$ on $\mathrm{SO}(2)$. As $\Bbb{C}$ and $\Bbb{C}'$ form a partition of $\Bbb{P}^1$ and $\Bbb{D}$, $\Bbb{D}c$, $\Bbb{D }'$ and $\Bbb{D}'c$ form a partition of $\mathrm{SO}(2)$ it follows that these are the only control sets of $S$ on $\Bbb{P}^1$ and $\mathrm{SO}(2)$, respectively. Considering the canonical projection $\pi:\mathrm{SO}(2)\to \Bbb{P}^1$ we have that
    \[
    \pi(\Bbb{D})=\pi(\Bbb{D}c)=\Bbb{C} ~\textrm{ and }~ \pi(\Bbb{D}')=\pi(\Bbb{D}'c)=\Bbb{C}'.
    \]
    As the order of $M$ is two and there are exactly two invariant control sets, on $\mathrm{SO}(2)$, we have that $C(S)=\{\mathrm{Id}\}$.
\end{exemplo}

\begin{exemplo}
    Let $G=\mathrm{Sl}(3,\R)$ be the group of ~$3\times 3$ matrices with entries in $\R$ and determinant 1. The group $G$ acts on the projective space $\Bbb{P}^2$ by $g[x]=[gx]$, where $[x]$ is the subspace generated by the non-zero vector $x \in \R^3$.
	
    We have that the isotropy subgroup of the point $[e_1] \in \Bbb{P}^2$, where $\mathcal{B}=\{e_1,e_2,e_3\}$ is the canonical basis for $\R^3$, is the subgroup $P$ of $G$ whose elements are written in the form
    \[
    \left(
    \begin{array}{ccc}
	* & * & * \\
	0 & * & * \\
	0 & * & * \\
    \end{array}
    \right).
    \]
    As the action of $G$ on $\Bbb{P}^2$ is transitive it follows that $\Bbb{P}^2\simeq G/P$. Let $\la_0$ be the Weyl chamber in $G$ of the diagonal matrices $h=\mathrm{diag}(a_1,a_2,a_3)$ such that $a_1>a_2>a_3>0$. This Weyl chamber determines an Iwasawa decomposition of $G$, which can be written as $G=\mathrm{SO}(3)AN$, where $A$ and $N$ are the subgroups of $G$ of the diagonal matrices with positive entries and the upper triangular matrices with ones on the diagonal, respectively. The centralizer of $A$ in $\mathrm{SO}(3)$ is the subgroup
    \[
    M=\left\{
    \mathrm{diag}(1,1,1), 
    \mathrm{diag}(1,-1,-1), 
    \mathrm{diag}(-1,1,-1), 
    \mathrm{diag}(-1,-1,1)
    \right\}.
    \]
    Taking $H=\mathrm{diag}(2,-1,-1)$ and denoting by $K_H$ the centralizer of $H$ in $\mathrm{SO}(3)$, the subgroup $P$ decomposes as $P=K_HAN$. Therefore, the fiber of the canonical projection $G/AN \to \Bbb{P}^2$ is diffeomorphic to the subgroup $K_H$. Note also that if $K_H^0$ denotes the identity component of $K_H$, then
    \[
    K_H^0=
    \left\{
    \left(
    \begin{array}{ccc}
	1 & 0 & 0 \\
	0 & \cos t & -\sin t \\
	0 & \sin t & \cos t \\
    \end{array}
    \right) : t \in \R
    \right\}
    \]
    and $K_H=K_H^0\cup (K_H^0k)$, where $k=\mathrm{diag}(-1,1,-1)$. In particular, $K_H$ has two connected components.
	
    Let $S = \mathrm{Sl}^+(3,\R)$ be the semigroup of matrices, in $G$, with positive entries. We have that $S$ is an open semigroup of $G$. Let $\Bbb{C}_{\Bbb{P}}$ be the subset of $\Bbb{P}^2$ defined by
    \[
    \Bbb{C}_{\Bbb{P}}=\{[(x_1,x_2,x_3)]:x_i\geq 0, ~\text{for all}~ i\}.
    \]
    Given $[x], [y] \in \mathrm{int}\Bbb{C}_{\Bbb{P}}$, $x=(x_1,x_2,x_3)$ and $y=(y_1,y_2,y_3)$, let
    \[
    g'=\sqrt[3]{\frac{x_1x_2x_3}{4y_1y_2y_3}}
    \left(
    \begin{array}{ccc}
	2y_1/x_1 & y_1/x_2 & y_1/x_3 \\
	y_2/x_1 & 2y_2/x_2 & y_2/x_3 \\
	y_3/x_1 & y_3/x_2 & 2y_3/x_3 \\
    \end{array}
    \right).
    \]
    We have $g' \in S$ and $g'[x]=[y]$. Therefore, $S$ is transitive on $\mathrm{int}\Bbb{C}_{\Bbb{P}}$. As for all $[x] \in \Bbb{C}_{\Bbb{P}}$, there exists $g \in S$ such that $g[x] \in \mathrm{int}\Bbb{C}_{\Bbb{P}}$ we have $\mathrm{cl}(S[x])=\Bbb{C}_{\Bbb{P}}$, for all $[x] \in \Bbb{C}_{\Bbb{P}}$. Therefore, $\Bbb{C}_{\Bbb{P}}$ is the invariant control set of $S$ in $\Bbb{P}^2$. Furthermore, the flag type of $S$ is the projective space $\Bbb{P}^2$ (see \cite{gfsm}, Corollary 3). Let $g \in G$ be such that $\la=g\la_0 g^{-1}$ intersects $S$. By Proposition 3.2 of \cite{sms} there are invariant control sets $\Bbb{D}$ and $\Bbb{D}'$, in $G/AN$, such that $\Bbb{D}$ is diffeomorphic to $\Bbb{C}_{\Bbb{P}}\times (gK_H^0g^{-1})$ and $\Bbb{D}'$ is diffeomorphic to $\Bbb{C}_{\Bbb{ P}}\times (gK_H^0kg^{-1})$ and, moreover, these are the only invariant control sets on $G/AN$. Therefore, the group $C(S)$, in this case, is conjugated, by $g$, of the subgroup
    \[
    \left\{
    \mathrm{diag}(1,1,1), 
    \mathrm{diag}(1,-1,-1)
    \right\}.
    \]
    The other groups $C(S,w)$ are conjugates of $C(S)$ by $w$. Since $S$ has 3 control sets on the maximal flag manifold of $\mathrm{Sl}(3,\R)$ (see \cite{smt}, Example 5.4) it follows that $S$ has 6 control sets on $\mathrm{SO}(3)$.
\end{exemplo}

\section{Acknowledgments}
La\'ercio dos Santos was partially supported by FAPEMIG RED-00133-21.

\end{document}